\documentclass[11pt]{article}

\usepackage{amsmath,amssymb,mathtools,bm}
\usepackage{amsthm}
\usepackage{graphicx,booktabs,float,longtable,array,multirow}
\usepackage{algorithm}
\usepackage{algpseudocode}
\usepackage{enumitem}
\usepackage{xcolor}
\usepackage[numbers,sort&compress]{natbib}
\usepackage[colorlinks=true,linkcolor=blue,citecolor=blue,urlcolor=blue]{hyperref}
\usepackage[margin=1in]{geometry}
\graphicspath{{figures/}}

\theoremstyle{plain}
\newtheorem{theorem}{Theorem}
\newtheorem{lemma}{Lemma}
\newtheorem{proposition}{Proposition}
\theoremstyle{remark}
\newtheorem{remark}{Remark}
\newcommand{\R}{\mathbb{R}}
\newcommand{\dd}{\,\mathrm d}
\newcommand{\norm}[1]{\left\lVert #1\right\rVert}

\newcommand{\eps}{\varepsilon}

\newcommand{\spanop}{\operatorname{span}}
\newcommand{\diag}{\operatorname{diag}}

\newcommand{\vc}{\boldsymbol{c}}
\newcommand{\vd}{\boldsymbol{d}}
\newcommand{\vb}{\boldsymbol{b}}
\newcommand{\vphi}{\boldsymbol{\varphi}}
\newcommand{\vgamma}{\boldsymbol{\gamma}}
\newcommand{\bx}{\boldsymbol{x}}

\newcommand{\ba}{\boldsymbol{a}}

\newcommand{\valpha}{{\boldsymbol{\alpha}}}
\newcommand{\vpsi}{\boldsymbol{\psi}}

\newcommand{\vxi}{\boldsymbol{\xi}}

\newcommand{\tableformat}{%
  \small
  \setlength{\tabcolsep}{4.5pt}%
  \renewcommand{\arraystretch}{1.08}%
}

\newcommand{\meanstd}[2]{\begin{tabular}[c]{@{}c@{}}$#1$\\[-1pt]$(#2)$\end{tabular}}

\begin{document}
\emergencystretch=3em

\title{Evo-GTransNet for Parabolic PDEs: A Fixed-Feature Galerkin Method of Lines with Quadrature-Mass Orthonormalization}
\author{Lili Ju\thanks{Department of Mathematics, University of South Carolina,
Columbia, SC 29208, USA. Email: \href{mailto:ju@math.sc.edu}{ju@math.sc.edu}.}
\and
Jin Zhao\thanks{Corresponding author. Academy for Multidisciplinary Studies,
Capital Normal University, Beijing 100048, China. Email:
\href{mailto:zjin@cnu.edu.cn}{zjin@cnu.edu.cn}.}}

\date{}
\maketitle

\begin{abstract}
In this paper, we develop an evolutionary generalized transferable neural network (Evo-GTransNet) solver for parabolic partial differential equations, formulated as a retained-space fixed-feature Galerkin method of lines. A GTransNet provides the prescribed spatial dictionary, while only the retained output coefficients evolve, thereby avoiding nonlinear training during time integration. To address severe mass-matrix ill-conditioning, we apply a quadrature-weighted truncated singular value decomposition (SVD) to select the numerically resolved trial space, followed by a separate rescaling that makes its basis orthonormal with respect to the assembly-quadrature mass inner product. Rank truncation modifies the approximation space, whereas the subsequent orthonormalization changes only its coordinate representation and preserves the retained discrete functions in exact arithmetic.
The resulting semidiscrete coefficient system has an identity mass matrix, and we establish a semidiscrete energy law for symmetric linear parabolic problems. With the implicit midpoint scheme for time discretization, we further prove the contractivity of the method and derive a conditional fully discrete error estimate in which the error is controlled by the retained-space approximation error and the consistency defects. Numerical experiments demonstrate second-order temporal convergence using repeated feature samples together with separate assembly and validation quadratures, and quantify the accuracy of the retained space in the presence of severe raw-mass ill-conditioning. For the high-frequency and multiscale benchmark problems considered here, GTransNet achieves the smallest mean validation errors among the tested fixed-feature dictionaries at the same nominal output dimension.
\end{abstract}



\section{Introduction}

Neural-network approximations provide flexible, meshfree trial spaces for numerical solution of partial differential equations (PDEs).  Beginning with early neural trial-function methods \cite{lagaris1998ann}, the field has developed through the deep Ritz method \cite{e2018deepritz}, the deep Galerkin method \cite{sirignano2018dgm}, physics-informed neural networks \cite{raissi2019pinn,karniadakis2021piml}, and related variational and collocation formulations.  Variational PINN formulations enforce tested residual moments, and hp-VPINNs \cite{kharazmi2021hpvpinn} further combine this weak-form viewpoint with domain decomposition and local $h/p$ refinement of test spaces.  Most of these approaches determine both hidden-layer and output-layer parameters by nonlinear optimization.  Such training can be expensive and optimization-sensitive, and composite physics-informed losses can additionally produce unbalanced back-propagated gradients and stiff training dynamics \cite{wang2021gradient}.  It is also known that neural approximations may favor low-frequency components \cite{xu2020fprinciple,rahaman2019spectralbias} over high-frequency ones of the solutions.

Prescribed Fourier features can modify this spectral behavior \cite{tancik2020fourier}. More generally, fixed-feature methods take a complementary approach by first constructing a trial dictionary and then solving only for the output coefficients. Random feature method (RFM) and extreme learning machine (ELM) are representative examples of this strategy \cite{rahimi2007random,huang2006elm,chen2022rfm,dong2021localelm}.
Transferable neural network (TransNet) \cite{zhang2024transnet} is a single-hidden-layer shallow neural network that replaces a fully random hidden layer with features associated with geometry-informed partition hyperplanes. Generalized TransNet (GTransNet) \cite{cheng2026gtransnet} retains this geometric construction in the first hidden layer and further enriches it by centrally symmetric bias sampling and variance-controlled additional hidden layers for steady-state PDEs. In both methods, the hidden-layer parameters are fixed after sampling, while the output coefficients are determined by solving a linear least-squares problem. This separation is attractive for high-frequency and multiscale approximation, but its extension from steady-state problems to initial-value problems is not straightforward. A time-dependent solver must propagate the initial data causally, enforce boundary conditions throughout the evolution, and remain numerically stable when the prescribed features contain nearly linearly dependent directions.

Existing neural solvers treat time-dependent problems in different ways. Global space--time methods approximate the spatial and temporal variables simultaneously and can be effective when combined with suitable partitions and weighting strategies \cite{chen2023rfmtime}. Evolutional deep neural networks and Neural Galerkin methods evolve neural representations sequentially in time \cite{du2021ednn,bruna2023ng}. A recent discrete-time random-feature method combines fixed spatial features with third-order IMEX--Runge--Kutta time stepping and stage-wise least-squares solves \cite{zhou2026discretetime}.
A related fixed reduced-space framework appears in POD--Galerkin methods for parabolic equations. However, POD constructs the reduced basis from solution snapshots, whereas the retained space considered here is generated from prescribed GTransNet features \cite{kunisch2001pod}.
Recent work has also combined feature spaces with Galerkin projection, explicit orthogonality control, and rank-revealing linear algebra \cite{tang2026dnngalerkin,jia2026orthogonal,tan2026preconditioned,vanbeek2026filtering}.  These studies indicate that feature construction, temporal discretization, and numerical-rank control are largely independent design choices.

In this paper, we adopt a fixed spatial-subspace viewpoint similar to that in the classical method-of-lines and reduced-basis formulations \cite{thomee2006galerkin,hesthaven2022reduced}, where the spatial dictionary is fixed after construction and only the output coefficients evolve in time. For a prescribed GTransNet dictionary, we refer to this retained-space realization as \emph{Evo-GTransNet}. It is a fixed-feature Galerkin method of lines for parabolic equations: a hard boundary factor, together with an optional boundary lifting, defines a boundary-conforming trial space, while only the retained output coefficients evolve.

Since the sampled neural features can be highly correlated, the quadrature mass matrix may be severely ill-conditioned. We therefore apply a rank-revealing truncated singular value decomposition (SVD) before time integration. This procedure removes numerically unresolved directions and constructs a basis that is orthonormal with respect to the assembly-quadrature mass matrix, so that the retained coefficient system has an identity mass matrix.
Galerkin projection, rank-revealing factorizations, mass-matrix orthonormalization, and implicit midpoint
time stepping are standard techniques in numerical analysis \cite{thomee2006galerkin,hesthaven2022reduced}. Another major novelty of the present work lies in their quadrature-consistent retained-space realization for redundant prescribed neural dictionaries, with GTransNet serving as the primary example. The proposed construction employs a retained space throughout the entire time interval, separates subspace truncation from coordinate orthonormalization, and distinguishes the assembly and validation inner products. Once the retained space has been constructed, the stability analysis depends only on the associated Galerkin matrices and not on the procedure used to generate the dictionary.


The rest of the paper is organized as follows. Section \ref{sec:background} reviews fixed-feature solvers, related methods, and the GTransNet dictionary. Sections \ref{sec:method}--\ref{sec:analysis} present the retained-space formulation, implementation algorithms, and the corresponding stability and error analysis, respectively. Section \ref{sec:numerics} reports numerical results that assess temporal convergence, the conditioning of the raw mass matrix, and the accuracy of the trial space. Finally, Section \ref{sec:conclusion} summarizes the main findings, discusses the limitations, and concludes the paper. Technical derivations and implementation details are provided in the Appendices.

\section{Background and motivations}\label{sec:background}

\subsection{Fixed-feature PDE solvers and rank control}

A neural feature or subspace solver approximates a scalar function $u$ by
\begin{equation}\label{appx}
  u_N(\bx)=\sum_{j=1}^N c_j \psi_j(\bx)=\vpsi(\bx)^\top\vc, \quad \bx\in\Omega,
\end{equation}
where $\Omega$ is the target spatial domain,  $\vpsi(\bx)=(\psi_1(\bx),\ldots,\psi_N(\bx))^\top$ is the prescribed final feature vector and $\vc=(c_1,\ldots,c_N)^\top$ contains the output coefficients obtained from a linear or linearized residual system.  ELMs freeze randomly sampled hidden parameters and determine only the output coefficients \cite{huang2006elm}; PIELM applies this output-only principle to stationary and time-dependent linear PDEs \cite{dwivedi2020pielm}.  Kernel random features originated as scalable approximations for kernel machines \cite{rahimi2007random}, whereas PDE-oriented output-only methods construct fixed trial spaces for collocation, partition-of-unity, time-dependent, interface, and local-subdomain formulations \cite{chen2022rfm,chen2023rfmtime,chi2024rfm,dong2021localelm}.

Fixed random neural trial functions have also been combined with Petrov--Galerkin test spaces or local discontinuous Galerkin coupling \cite{shang2023rnnpg,sun2024lrnndg}. Their reported methods for time-dependent problems use global space--time formulations rather than advancing a fixed spatial Galerkin system through a method-of-lines approach. Trained neural-subspace methods follow a different strategy: they first learn problem-dependent neural basis functions and then compute an approximate solution in their linear span \cite{xu2025subspace}. TransNet is a geometry-informed shallow neural network and distributes partition hyperplanes approximately uniformly over the computational domain \cite{zhang2024transnet}. Our proposed Evo-GTransNet, in contrast, constructs a prescribed spatial dictionary without problem-specific nonlinear training.
The main advantage of output layer-only feature solvers is that they avoid nonlinear training of the hidden representation. For a linear PDE, the method leads directly to a linear system for the output coefficients. For a nonlinear PDE, a Picard or Newton iteration still updates only the output coefficients. The main challenge is to construct a spatial dictionary that is sufficiently expressive for the target solution class while remaining numerically stable under the required linear algebra operations.

A particularly closely related recent work is the feature-based DNN Galerkin framework of Tang et al.~\cite{tang2026dnngalerkin}. It combines neural trial spaces with Galerkin projection, quadrature-based truncated SVD orthogonalization, a semidiscrete energy law, and energy-stable time integrators for gradient flows. The framework accommodates static random or structured bases, offline-pretrained frozen bases, and time-local bases reconstructed at prescribed restart times. While it supports both fixed and adaptive trial spaces, the present method applies a hard-boundary construction to a prescribed GTransNet dictionary, selects a quadrature-resolved retained subspace once before time integration, and keeps that subspace fixed over the entire time interval. Within this setting, we distinguish subspace truncation from coordinate orthonormalization and develop the corresponding analysis for fixed-feature space parabolic PDE solvers.
Zhou et al.~\cite{zhou2026discretetime} also use a fixed random-feature trial space, but compute the solution at each stage by solving a collocation least-squares problem within a four-stage, third-order IMEX--Runge--Kutta scheme. In contrast, our method advances a single weak-form Galerkin coefficient system in coordinates that are orthonormal with respect to the assembly-quadrature mass inner product.

Complementary studies have addressed conditioning and redundancy at both the feature-construction and algebraic-solver levels. Jia and Wang \cite{jia2026orthogonal} introduce orthogonality regularization during physics-informed feature pretraining and study the effective rank, conditioning, and transferability of the resulting features. Tan and Chen \cite{tan2026preconditioned} combine CountSketch compression with QR- or SVD-based right preconditioning for overdetermined random-feature least-squares systems, whereas van Beek et al.~\cite{vanbeek2026filtering} use local block rank-revealing QR factorizations to filter redundant features and precondition domain-decomposed systems. These approaches operate, respectively, during feature learning, in the global least-squares solve, and within local domain-decomposition blocks. In our method, by contrast, numerical-rank selection is defined through the quadrature-mass inner product associated with a boundary-conforming weak-form trial space, and the resulting retained subspace is fixed throughout the time integration.

\subsection{TransNet and GTransNet}

Assume that the domain $\Omega\subset B_R(\bx_c)\subset\mathbb{R}^d$, where $B_R(\bx_c)$ denotes the ball centered at $\bx_c$ with radius $R>0$. The single-hidden-layer features in TransNet are defined by
\begin{equation}
\psi_j^{[1]}(\bx)=\tanh\left(\gamma_j(\ba_j^\top(\bx-\bx_c)+R\xi_j)\right),\quad j=1,2,\cdots,N,
\label{eq:trans_feature}
\end{equation}
where $\ba_j\in\mathbb R^{d}$ is a randomly sampled unit normal and $\xi_j\stackrel{\rm iid}{\sim}\mathcal U[0,1]$ determines the offset of the partition hyperplane relative to the oriented normal, and $\gamma_j>0$ is a shape parameter. The key idea of TransNet is to distribute the partition hyperplanes, i.e., $\{\ba_j^\top(\bx-\bx_c)+R\xi_j=0\}_{j=1}^N$, approximately uniformly over the domain through a dedicated sampling algorithm for $\{\ba_j\}_{j=1}^N$ and $\{\xi_j\}_{j=1}^N$ (see \cite{zhang2024transnet} for details). 
Several recent TransNet variants address complementary spatial challenges. Multi-TransNet assigns distinct TransNet dictionaries to interface subdomains and couples them through interface conditions \cite{lu2025multitransnet}. MAE-TransNet combines separate TransNet approximations of the matched asymptotic inner and outer solutions for singularly perturbed boundary-layer problems \cite{shen2026maetransnet}. Weak TransNet uses TransNet features as trial functions, radial basis functions as test functions, and determines the output coefficients by minimizing the weak residual in a least-squares sense \cite{xu2026weaktransnet}. 
These methods perform very well for problems with smooth or piecewise smooth solutions, but their performance can still deteriorate when large values of $\gamma_j$ drive the activation functions toward saturation.

GTransNet extends this construction by adding additional hidden layers. Let $N_\ell$ denote the width of hidden layer $\ell$, and let $L$ be the total number of hidden layers. For the first hidden layer width $N_1$, it inherits the TransNet construction with slight changes and is defined by
\begin{equation}
  \vpsi^{[1]}(\bx)
  =\tanh\left(
    \mathbf\Gamma\bigl(\mathbf A^{[1]}(\bx-\bx_c)+R\vxi\bigr)
  \right)\in\R^{N_1},
  \label{eq:gtrans_first}
\end{equation}
where the rows of $\mathbf A^{[1]}\in\R^{N_1\times d}$ are independently sampled unit vectors, $\vxi=(\xi_1,\ldots,\xi_{N_1})^\top$ with $\xi_i\stackrel{\rm iid}{\sim}\mathcal U[-1,1]$, ${\vpsi}^{[1]}=(\psi_{1}^{[1]},\cdots,\psi_{N_1}^{[1]})$, $\mathbf\Gamma:=\diag(\vgamma)$, and $\vgamma=(\gamma_1,\ldots,\gamma_{N_1})^\top$.
The symmetry of $\{\xi_i\}_{i=1}^{N_1}$ refers to the sampling distribution (finite samples are not required to occur in symmetric pairs). In all experiments reported in this paper, we simply  take $\mathbf\Gamma=\gamma\mathbf I_{N_1}$, i.e., $\gamma_i\equiv\gamma$ for $i=1,2,\cdots, N_1$. Subsequent bias-free hidden layers are defined recursively by
\begin{equation}
  \vpsi^{[\ell]}(\bx)
  =\tanh\left(\mathbf W^{[\ell]}\vpsi^{[\ell-1]}(\bx)\right)
  \in\R^{N_\ell},
  \qquad 2\le\ell\le L,
  \label{eq:gtrans_second}
\end{equation}
where
\(\mathbf W^{[\ell]}\in\mathbb R^{N_\ell\times N_{\ell-1}}\)
is sampled independently according to
\begin{equation}\label{eq:variance_control}
  (\mathbf W^{[\ell]})_{ij}
  \sim
  \mathcal N\left(0,\frac{\delta}{N_{\ell-1}}\right),
  \qquad 2\le\ell\le L.
\end{equation}
with a prescribed constant $0<\delta\le1$ as prescribed in \cite{cheng2026gtransnet}.  
The activation $\tanh$ is applied componentwise. The variance-controlled hidden layers are intended to keep the activations away from saturation while generating richer multiscale features.
The final GTransNet feature vector and its scalar components are
\begin{equation}
  \vpsi(\bx):=\vpsi^{[L]}(\bx),
  \qquad
  \psi_j(\bx):=\psi_j^{[L]}(\bx),
  \qquad j=1,\ldots,N_L.
\end{equation}
 In the two-hidden-layer architecture used throughout this paper, we have $L=2$ and  $N_2=N$, and  the $N$ components of $\vpsi^{[2]}$ form the final feature set to be used in \eqref{appx}.
Thus $\psi_j=\psi_j^{[1]}$ for TransNet and $\psi_j=\psi_j^{[2]}$ for the GTransNet architecture here.  

Most existing TransNet/GTransNet methods are developed for steady-state elliptic, interface, or singularly perturbed problems as mentioned above and do not address the fixed-feature space evolution of the Galerkin coefficient system for parabolic PDE problems considered in this paper.
For a steady-state problem, a single residual least-squares system is a natural formulation. For an initial-value problem, however, the numerical method must respect causality and stability. Evolutional deep neural networks (EDNNs) \cite{du2021ednn} treat the neural-network parameters as time-dependent variables after fitting the initial condition. Neural Galerkin methods \cite{bruna2023ng} evolve the neural representation through sequential-in-time residual projection and active learning. Randomized sparse Neural Galerkin methods \cite{berman2023sparse} further demonstrate that evolving a large number of parameters can be computationally expensive and may amplify error accumulation over time.
These observations motivate a conservative extension of GTransNet: the sampled hidden features are kept fixed, and only the output coefficients evolve. The resulting formulation retains the prescribed GTransNet feature construction while introducing a time-causal Galerkin evolution.

\section{The Evo-GTransNet method}\label{sec:method}

Let $\Omega\subset\R^d$ be a bounded Lipschitz domain and let the terminal time $T>0$. Although Evo-GTransNet can be applied to various types of parabolic PDEs, as shown in the numerical experiments, the model problem considered here for algorithm design and analysis is the following linear diffusion equation with a homogeneous Dirichlet boundary condition:
\begin{equation}
  \begin{cases}
  u_t-\nabla\cdot(a(\bx,t)\nabla u)=f(\bx,t),
    & (\bx,t)\in\Omega\times(0,T],\\
  u(\bx,t)=0,
    & (\bx,t)\in\partial\Omega\times(0,T],\\
  u(\bx,0)=u_0(\bx),
    & \bx\in\Omega.
  \end{cases}
  \label{eq:linear_parabolic}
\end{equation}
Let us assume that the scalar diffusion coefficient is measurable and uniformly positive and bounded,
$0<a_0\le a(\bx,t)\le a_1<\infty$, and that the data have sufficient regularity for the formulations under consideration.  For almost every $t\in(0,T]$, the weak problem is to find $u(t)\in H_0^1(\Omega)$ such that
\begin{equation}
  \langle u_t,v\rangle
  +\int_\Omega a(\bx,t)\nabla u\cdot\nabla v\dd \bx
  =\int_\Omega f(\bx,t)v\dd \bx
  \qquad\forall\,v\in H_0^1(\Omega),
  \label{eq:continuous_weak_form}
\end{equation}
with $u(0)=u_0$.  Here $\langle\cdot,\cdot\rangle$ denotes the $H^{-1}(\Omega)$--$H_0^1(\Omega)$ duality pairing, or the $L^2(\Omega)$ inner product when $u_t\in L^2(\Omega)$.  This weak formulation is the standard parabolic Galerkin framework; GTransNet enters through the construction of the finite-dimensional trial space.


\subsection{Boundary-conforming fixed-feature space and Galerkin method of lines}

Let $\{\psi_j\}_{j=1}^N$ be a prescribed GTransNet dictionary (or another baseline feature dictionary of the type described in Section~\ref{sec:background}). Let $D\in W^{1,\infty}(\Omega)\cap C^1(\overline\Omega)$ have zero trace on $\partial\Omega$. Since the tanh features used here are smooth and bounded, the hard-boundary functions
\begin{equation}
  \varphi_j(\bx)=D(\bx)\psi_j(\bx),
  \label{eq:raw_trial_space}
\end{equation}
define a boundary-conforming  trial space
\begin{equation}
  V_N=\spanop\{\varphi_1,\ldots,\varphi_N\}
  \subset H_0^1(\Omega).
  \label{eq:raw_trial_space1}
\end{equation}
The hidden features and $V_N$ are fixed after sampling, and only the output coefficients depend on time, i.e., for the homogeneous problem \eqref{eq:linear_parabolic}, the approximation is
\begin{equation}
  u_N(\bx,t)
  =\sum_{j=1}^N c_j(t)\varphi_j(\bx)
  =\vphi(\bx)^\top\vc(t).
  \label{eq:eg_ansatz}
\end{equation}
where  $\vphi(\bx):=(\varphi_1(\bx),\ldots,\varphi_N(\bx))^\top$ and $\vc(t):=(c_1(t),\ldots,c_N(t))^\top$.
In the numerical experiments, for $\Omega=(0,1)$, $\Omega=(0,1)^2$, or $\Omega=(0,1)^3$,  we use the corresponding tensor-product boundary factors:
\[
\begin{aligned}
  D(x)&=x(1-x),\\
  D(x,y)&=x(1-x)y(1-y),\\
  D(x,y,z)&=x(1-x)y(1-y)z(1-z).
\end{aligned}
\]
For general geometries, approximate-distance and R-function constructions provide a systematic approach to constructing trial functions that satisfy the essential boundary data exactly \cite{sukumar2022exactbc}.


Restricting \eqref{eq:continuous_weak_form} to $V_N$ gives the semidiscrete (in space) problem: find $u_N(t)\in V_N$ such that
\begin{equation}
  \int_\Omega \partial_tu_N\,v_N\dd \bx
  +\int_\Omega a(\bx,t)\nabla u_N\cdot\nabla v_N\dd \bx
  =\int_\Omega f(\bx,t)v_N\dd \bx,
  \qquad \forall\,v_N\in V_N.
  \label{eq:discrete_weak_form}
\end{equation}
Substituting \eqref{eq:eg_ansatz} and testing successively with
$v_N=\varphi_i$ yields the method-of-lines Galerkin system
\begin{equation}
  \mathbf M\dot{\vc}(t)+\mathbf K(t)\vc(t)=\vb(t),
  \qquad 0<t\le T,
  \label{eq:gal_ode}
\end{equation}
where 
\begin{equation}
\begin{aligned}
  M_{ij}&=\int_\Omega\varphi_j\varphi_i\dd \bx,\\
  K_{ij}(t)&=\int_\Omega a(\bx,t)\nabla\varphi_j\cdot\nabla\varphi_i\dd \bx,\\
  b_i(t)&=\int_\Omega f(\bx,t)\varphi_i\dd \bx.
\end{aligned}
  \label{eq:raw_galerkin_matrices}
\end{equation}
Thus the neural construction supplies a fixed Galerkin basis, while the evolution itself is a standard finite-dimensional parabolic system.  The weak form is preferable to strong-form collocation for diffusion: it uses only first spatial derivatives and gives a symmetric positive-semidefinite stiffness matrix when $a>0$, whereas strong-form collocation generally requires second derivatives and does not preserve this symmetry.

The raw mass matrix in \eqref{eq:gal_ode} can be severely ill-conditioned because sampled features may be nearly linearly dependent.  Section~\ref{sec:normalization} therefore uses a quadrature-weighted SVD to select a resolved retained subspace and construct identity-mass coordinates.  The implicit midpoint rule is then applied to that retained system; 
the ill-conditioned raw system is not advanced directly.

\begin{remark}[Nonhomogeneous Dirichlet boundary condition]
The analysis and implementation reported in this paper use homogeneous Dirichlet data.  Formally, nonhomogeneous data with the boundary trace $g$ can be incorporated through a sufficiently regular lifting $u_b$ and
\[
  u_N(\bx,t)=u_b(\bx,t)+\vphi(\bx)^\top\vc(t),
\]
with the compatibility condition
$u_b(\cdot,0)|_{\partial\Omega}=u_0|_{\partial\Omega}$.  If $a\nabla u_b$ has sufficient regularity, the weak formulation \eqref{eq:continuous_weak_form} 
retains the same left-hand side and takes the following right-hand side
\begin{equation}
  b_i^{\rm lift}(t) = 
  \int_\Omega f(\bx,t)\varphi_i\dd \bx
   -\int_\Omega u_{b,t}(\bx,t)\varphi_i\dd \bx
   -\int_\Omega a(\bx,t)\nabla u_b\cdot\nabla\varphi_i\dd \bx.
  \label{eq:lifted_load}
\end{equation}
The initial datum is then projected after replacing $u_0$ by
$u_0-u_b(\cdot,0)$.  A nonhomogeneous lifting enters the corresponding stability and error estimates as an additional forcing contribution.\label{nonhomo}
\end{remark}

\begin{remark}[Nonlinear evolution equations]
For a general evolution equation $u_t=\mathcal N(u,t)$, projection onto the same fixed-feature  space $V_N$ formally gives
$$
  \mathbf M\dot{\vc}=\boldsymbol F(t,\vc),$$
  where $\boldsymbol F(t,\vc)$ comes from spatial discretization of $\mathcal N(u,t)$.
The Allen--Cahn problem tests in Section~\ref{sec:numerics} further uses a separate stabilized nonlinear update on the retained fixed-feature space.
\end{remark}

\subsection{Weighted-SVD rank selection and quadrature-mass orthonormalization}\label{sec:normalization}


Let $\{\bx_q,w_q\}_{q=1}^Q$ be quadrature points with strictly positive weights $w_q>0$ on the domain $\Omega$, and let $0<\tau<1$ be a prescribed eigenvalue threshold. For functions $v$ and $z$ with well-defined nodal values, define the positive semidefinite quadrature bilinear form and its associated seminorm by
$$
    (v,z)_q:=\sum_{m=1}^Qw_mv(\bx_m)z(\bx_m)
  =\mathbf v_q^\top\mathbf W_q\mathbf z_q,
  \qquad \norm{v}_q^2=(v,v)_q,
$$
where $\mathbf v_q$ and $\mathbf z_q$ denote the corresponding nodal vectors, and $\mathbf W_q=\diag(w_1,\ldots,w_Q)$. Define
$
  (\bm\Phi)_{qj}=\varphi_j(\bx_q),
$
for $q=1,\ldots,Q,\;
   j=1,\ldots,N$, then the discrete mass matrix is given by
\begin{equation}
  \mathbf M_q=\bm\Phi^\top\mathbf W_q\bm\Phi.
\end{equation}
Let
\begin{equation}
  \mathbf B:=\mathbf W_q^{1/2}\bm\Phi
  =\mathbf U\bm\Sigma\mathbf V^\top
  \label{SVD}
\end{equation}
be a thin singular value decomposition, where $\bm\Sigma=\diag(\sigma_1,\ldots,\sigma_N)$ contains the singular values of $\mathbf B$ in nonincreasing order. It follows that
$
\mathbf M_q=\mathbf V\bm\Sigma^2\mathbf V^\top.
$

Assume that the sampled mass matrix is nonzero, then $\sigma_1>0$ and  at least one nonzero singular direction is remained. 
Our  implementation computes this SVD and retains the directions whose equivalent eigenvalues of ${\mathbf M}_q$, i.e., $\{\lambda_i=\sigma_i^2\}_{i=1}^{N}$, satisfy
 $ \lambda_i>\tau\lambda_{1}.$
Equivalently, in singular-value notation this is $\sigma_i>\sqrt{\tau}\sigma_1$. Suppose that a total of $r$ singular values are retained. Let $\mathbf V_r$ denote the retained right singular vectors and $\bm\Sigma_r$ the corresponding diagonal matrix of retained singular values. We define
\begin{equation}
  \mathbf T_r:=\mathbf V_r\bm\Sigma_r^{-1},
  \qquad
  \widetilde{\bm\Phi}:=\bm\Phi\mathbf T_r.
  \label{eq:normalized_phi}
\end{equation}  
Spectral truncation is a classical regularization and numerical-rank device \cite{hansen1987tsvd}.  Here it is applied to the quadrature-weighted feature matrix to define a retained function space, rather than merely to regularize one coefficient solve.

The matrix $\widetilde{\bm\Phi}$ is not a newly trained dictionary; its columns are retained linear combinations of the original hard-boundary features, chosen to be orthonormal with respect to the quadrature mass inner product. This is the standard weighted-SVD orthogonalization applied to the prescribed trial space \cite{hesthaven2022reduced}. At the function level, define $\widetilde\vphi(\bx):=(\widetilde\varphi_1(\bx),\ldots,\widetilde\varphi_r(\bx))^\top$ by
\begin{equation}
  \widetilde{\vphi}(\bx):=\mathbf T_r^\top\vphi(\bx),
  \qquad
  \widetilde\varphi_i(\bx)=\sum_{j=1}^N(\mathbf T_r)_{ji}\varphi_j(\bx).
  \label{eq:retained_basis}
\end{equation}
Define the retained trial space by
$$
  V_{N,r}
  :=\spanop\{\widetilde\varphi_1,\ldots,\widetilde\varphi_r\}
  \subset V_N.
$$
Since $\mathbf T_r$ forms only linear combinations, each $\widetilde\varphi_i$ has the same homogeneous boundary trace as the original functions $\varphi_j$. From this point onward, let $\valpha(t)=(\alpha_1(t),\ldots,\alpha_r(t))^\top\in\R^r$ denote the retained normalized coefficient vector, then $\vc(t)=\mathbf T_r\valpha(t)\in\R^N$ gives its representation in the original dictionary. Thus, we have 
\begin{equation}
  u_{N,r}(\bx,t)
  :=\widetilde{\vphi}(\bx)^\top\valpha(t)
  =\sum_{i=1}^r\alpha_i(t)\widetilde\varphi_i(\bx).
  \label{eq:retained_ansatz}
\end{equation}
At the quadrature nodes, the two coordinate representations satisfy
 $ \widetilde{\bm\Phi}\valpha=\bm\Phi\vc
  =\bm\Phi\mathbf T_r\valpha.$
For brevity, we still write $u_N$ for the retained approximation in the numerical experiments whenever no ambiguity arises.

For the stiffness terms, let $\mathbf G_\ell\in\R^{Q\times N}$ denote the raw derivative-value matrix with entries
\begin{equation}
  (\mathbf G_\ell)_{qj}=\partial_{x_\ell}\varphi_j(\bx_q),
  \qquad q=1,\ldots,Q,\;
   j=1,\ldots,N.
\end{equation}
These entries are the spatial derivatives of the prescribed features. In the implementation, they are evaluated analytically by the chain rule for the tanh feature maps  of the fixed network, rather than by optimization or backpropagation during training. The formulas used for derivative assembly are given in \ref{app:feature_derivatives}. The normalized derivative matrix is
\begin{equation}
  \widetilde{\mathbf G}_\ell
  :=\mathbf G_\ell\mathbf T_r,
  \qquad \ell=1,\ldots,d,
  \label{eq:normalized_derivative}
\end{equation}
since $\mathbf T_r$ is independent of the spatial variable.

\begin{remark}
Equivalently, one may diagonalize $\mathbf M_q=\mathbf V\bm\Lambda\mathbf V^\top$, retain the eigenvalues satisfying $\lambda_i>\tau\lambda_{1}$, and define
\begin{equation}
  \widetilde{\bm\Phi}
  =\bm\Phi\mathbf V_r\bm\Lambda_r^{-1/2}.
  \label{eq:eig_phi}
\end{equation}
The numerical implementation uses the SVD formulation in practice because it yields smaller mass-orthogonality defects for near-threshold directions. 
Appendix~\ref{app:threshold} states the threshold convention used in the numerical experiments, and Appendix~\ref{app:scaling} summarizes the resulting dense linear-algebra costs.
\end{remark}

By using the above weighted SVD construction, we have 
\begin{equation}
  \mathbf W_q^{1/2}\widetilde{\bm\Phi}
  =\mathbf W_q^{1/2}\bm\Phi\mathbf T_r
  =\mathbf U_r\bm\Sigma_r\mathbf V_r^\top
   \mathbf V_r\bm\Sigma_r^{-1}
  =\mathbf U_r,
\end{equation}
then 
\begin{equation}
  \widetilde{\bm\Phi}^{\top}\mathbf W_q\widetilde{\bm\Phi}
  =(\mathbf W_q^{1/2}\widetilde{\bm\Phi})^\top
   (\mathbf W_q^{1/2}\widetilde{\bm\Phi})
  =\mathbf U_r^\top\mathbf U_r
  =\mathbf I_r,
\end{equation}
 thus the following properties hold. 
\begin{proposition}[Quadrature-mass orthonormality]\label{prop:mass}
For the quadrature rule and retained SVD subspace defined above, the retained basis satisfies, in exact arithmetic,
\begin{equation}
  (\widetilde\varphi_i,\widetilde\varphi_j)_q=\delta_{ij},
  \qquad\text{equivalently}\qquad
  \widetilde{\bm\Phi}^{\top}\mathbf W_q\widetilde{\bm\Phi}=\mathbf I_r. 
\end{equation}
Consequently,
\begin{equation}
  \norm{\textstyle\sum_{i=1}^r z_i\widetilde\varphi_i}_q
  =\norm{\mathbf z}_2,
  \qquad \forall\,\mathbf z=(z_1,\ldots,z_r)^\top\in\R^r.
\end{equation}
\end{proposition}
We note that this is a quadrature identity; it does not by itself imply continuous $L^2$ orthonormality.  On $V_{N,r}$, $(\cdot,\cdot)_q$ is an inner product and $\|\cdot\|_q$ is a norm.  On a general function class, they remain a bilinear form and a seminorm because nodal values need not distinguish different functions.

\subsection{Raw-to-normalized matrices and transformed system}\label{sec:transformed_system}
Let us introduce
\begin{align*}
  \mathbf D_{a,q}(t)
  &:=\diag\bigl(a(\bx_1,t),\ldots,a(\bx_Q,t)\bigr),
  &\mathbf K_{\rm raw}(t)&
  :=\sum_{\ell=1}^d
     \mathbf G_\ell^\top\mathbf W_q
     \mathbf D_{a,q}(t)\mathbf G_\ell,\\
  \mathbf f_q(t)
  &:=(f(\bx_1,t),\ldots,f(\bx_Q,t))^\top,
  &\vb_{\rm raw}(t)
  &:=\bm\Phi^\top\mathbf W_q\mathbf f_q(t).
\end{align*}
Then the coordinate map gives the retained operators
\begin{align}
  \widetilde{\mathbf K}_q(t)
  &:=\mathbf T_r^\top\mathbf K_{\rm raw}(t)\mathbf T_r
    =\sum_{\ell=1}^d
      \widetilde{\mathbf G}_\ell^\top\mathbf W_q
      \mathbf D_{a,q}(t)\widetilde{\mathbf G}_\ell,
  \notag\\
  \widetilde{\vb}_q(t)
  &:=\mathbf T_r^\top\vb_{\rm raw}(t)
    =\widetilde{\bm\Phi}^\top\mathbf W_q\mathbf f_q(t).
  \label{eq:raw_to_normalized}
\end{align}
Replacing the integrals in \eqref{eq:gal_ode} by the chosen quadrature rule gives the raw quadrature-assembled system
\[
  \mathbf M_q\dot{\vc}(t)
  +\mathbf K_{\rm raw}(t)\vc(t)
  =\vb_{\rm raw}(t).
\]
Restricting $\vc=\mathbf T_r\valpha$ and testing this system with the columns of $\mathbf T_r$ yields the retained identity-mass system
\begin{equation}
  \dot{\valpha}(t)+\widetilde{\mathbf K}_q(t)\valpha(t)
  =\widetilde{\vb}_q(t),
  \qquad 0<t\le T,
  \label{eq:retained_ode}
\end{equation}
where $\mathbf T_r^\top\mathbf M_q\mathbf T_r=\mathbf I_r$ by Proposition~\ref{prop:mass} and \eqref{eq:raw_to_normalized} gives the retained stiffness and load. Thus the retained evolution does not require solving with the raw mass matrix.
Define the sampled initial datum by
\begin{equation}
  \mathbf u_{0,q}:=(u_0(\bx_1),\ldots,u_0(\bx_Q))^\top,
\end{equation}
then
$  \valpha(0)
  =\widetilde{\bm\Phi}^\top\mathbf W_q\mathbf u_{0,q},$
When a lifting is used in the case of nonhomogeneous Dirichlet boundary condition discussed in Remark~\ref{nonhomo}, we replace $\mathbf u_{0,q}$ by the nodal vector of $u_0-u_b(\cdot,0)$.  Stiffness matrices, load vectors, and initial data are therefore assembled in the same retained identity-mass coordinates.
The subscript $q$ indicates quadrature assembly, whereas the tilde indicates representation in the retained identity-mass coordinates. We retain the notation $\widetilde{\mathbf K}_q(t)$ and $\widetilde{\vb}_q(t)$ in the analysis and in displayed numerical schemes to distinguish these operators from the exact raw-coordinate operators $\mathbf K(t)$ and $\vb(t)$.

\begin{remark}[Coordinate invariance within the retained subspace]\label{rem:coordinate_invariance}
The weighted SVD performs two distinct operations: the threshold selects the retained right-singular subspace, while the scaling by $\bm\Sigma_r^{-1}$ orthonormalizes the coordinates within that subspace. 
The approximation introduced by the weighted-SVD construction comes entirely from the rank selection; finite-precision arithmetic may further distinguish implementations that are algebraically equivalent.
\end{remark}

\begin{remark}
We distinguish the full raw condition estimate,
$$
  \kappa_{\rm full}=\operatorname{cond}_2(\mathbf M_q),
$$
from the retained spectral ratio,
$$
  \kappa_{\rm ret}
  =\operatorname{cond}_2(\mathbf V_r^\top\mathbf M_q\mathbf V_r)
  =\lambda_{1}/\lambda_{\min,\rm kept}.
$$
The latter is bounded above by $\tau^{-1}$ by construction and measures the conditioning of the raw mass quadratic form restricted to the retained right-singular subspace. The full-column-rank coordinate transformation satisfies
$$
  \kappa_2(\mathbf T_r):=\norm{\mathbf T_r}_2\norm{\mathbf T_r^\dagger}_2
  =\sqrt{\kappa_{\rm ret}}.
$$
In double-precision arithmetic, the full condition estimate is no longer quantitatively meaningful once the smallest singular value is at or below $\epsilon_{\rm machine}$ times the largest. Several raw mass matrices in the main experiments fall into this regime, while the remaining ones are still severely ill-conditioned. Accordingly, Table~\ref{tab:diagnostics} in numerical experiments reports resolution categories rather than numerical condition estimates for the unresolved cases. By contrast, $\kappa_{\rm ret}$ is of order $10^{11}$ and is computed from the retained singular directions. Thus, the threshold determines the effective rank, while the subsequent coordinate scaling produces an identity-mass evolution on the retained subspace.
\end{remark}

\subsection{Time stepping along with the fixed-feature space}\label{sec:time_stepping}

For the model linear diffusion problem \eqref{eq:linear_parabolic}, restricting the quadrature Galerkin residual to the weighted-SVD retained space via the fixed coordinate transformation $\vc=\mathbf T_r\valpha$ yields the identity-mass semidiscrete initial-value problem \eqref{eq:retained_ode}. The sampled dictionary, quadrature rule, and transformation matrix $\mathbf T_r$ are constructed once before time stepping and remain fixed throughout the time interval $0\le t\le T$.
Let $0=t_0<t_1<\cdots<t_{N_t}=T$ be a uniform partition with time step size $\Delta t=T/N_t$, and let $t_{n+1/2}=(t_n+t_{n+1})/2$. Set $\widetilde{\mathbf K}_q^{n+1/2}:=\widetilde{\mathbf K}_q(t_{n+1/2})$ and
$\widetilde{\vb}_q^{n+1/2}:=\widetilde{\vb}_q(t_{n+1/2})$. For simplicity of illustration, we apply the implicit midpoint rule to the time discretization \eqref{eq:retained_ode}, which yields the following fully discrete system:
\begin{equation}
\begin{aligned}
  \left(\mathbf I_r+\frac{\Delta t}{2}\widetilde{\mathbf K}_q^{n+1/2}\right)\valpha^{n+1}
  &=\left(\mathbf I_r-\frac{\Delta t}{2}\widetilde{\mathbf K}_q^{n+1/2}\right)\valpha^n
  +\Delta t\,\widetilde{\vb}_q^{n+1/2},\\
  &\hspace{7em} n=0,1,\ldots,N_t-1.
\end{aligned}
  \label{eq:im_retained}
\end{equation}
with $\valpha^{0}=\valpha(0)$.
Finally, the fully discrete solution is given by $u_{N,r}^n(\bx)=\widetilde{\vphi}(\bx)^\top\valpha^n$ at $t=t_n$.
For an autonomous linear system, this update is algebraically identical to the classical Crank--Nicolson (or trapezoidal) scheme \cite{crank1947practical}. For time-dependent $\widetilde{\mathbf K}_q(t)$ or $\widetilde{\vb}_q(t)$, however, \eqref{eq:im_retained} differs slightly from the classical Crank--Nicolson scheme because both quantities are evaluated at the midpoint $t_{n+1/2}$. When $\widetilde{\mathbf K}_q^{n+1/2}$ is symmetric positive semidefinite and $\Delta t>0$, the matrix on the left-hand side of \eqref{eq:im_retained} is symmetric positive definite. Therefore, each time step is uniquely solvable. Algorithm~\ref{alg:evo_gtransnet_im} summarizes the complete Evo-GTransNet procedure for solving \eqref{eq:linear_parabolic}.

\begin{algorithm}[H]
\caption{The Evo-GTransNet Method for Solving \eqref{eq:linear_parabolic}
}
\label{alg:evo_gtransnet_im}
\begin{algorithmic}[1]
\Require The PDE data $a,f,u_0,T,\Delta t$; a hard-boundary factor $D$; a positive-weight quadrature rule; and a prescribed eigenvalue threshold $0<\tau<1$.
\Ensure $\{\valpha^n,u_{N,r}^n\}_{n=0}^{N_t}$.
\State Generate the raw GTransNet  dictionary $\{\psi_j\}_{j=1}^N$.
\State Form $\varphi_j=D\psi_j$, evaluate $\bm\Phi$ and $\mathbf G_\ell$ at the quadrature nodes, and set $\mathbf B=\mathbf W_q^{1/2}\bm\Phi$.
\State Compute the thin SVD $\mathbf B=\mathbf U\bm\Sigma\mathbf V^\top$.  If $\sigma_1=0$, terminate; otherwise retain $\mathbf U_r$, $\bm\Sigma_r$, and $\mathbf V_r$ satisfying $\sigma_i>\sqrt{\tau}\,\sigma_1$.
\State Set $\mathbf T_r=\mathbf V_r\bm\Sigma_r^{-1}$, $\widetilde{\bm\Phi}=\bm\Phi\mathbf T_r$, and $\widetilde{\mathbf G}_\ell=\mathbf G_\ell\mathbf T_r$ for $\ell=1,\ldots,d$.
\State Set $\valpha^0=\widetilde{\bm\Phi}^\top\mathbf W_q\mathbf u_{0,q}$ and reconstruct $u_{N,r}^0=\sum_{i=1}^r\alpha_i^0\widetilde\varphi_i$.
\For{$n=0,\ldots,N_t-1$}
  \State Set $t_{n+1/2}=(n+\tfrac12)\Delta t$ and assemble $\widetilde{\mathbf K}_q^{n+1/2}$ and $\widetilde{\vb}_q^{n+1/2}$.
  \State Solve \eqref{eq:im_retained} for $\valpha^{n+1}$.
  \State Reconstruct $u_{N,r}^{n+1}=\sum_{i=1}^r\alpha_i^{n+1}\widetilde\varphi_i$.
\EndFor
\end{algorithmic}
\end{algorithm}

All sampled features, $\mathbf T_r$, $\widetilde{\bm \Phi}$, and $\widetilde {\mathbf G}_\ell$ remain fixed throughout the time loop.  For a time-dependent diffusion coefficient or source, only $\widetilde{\mathbf K}_q^{n+1/2}$ or $\widetilde {\vb}_q^{n+1/2}$ is reassembled at the midpoint.  The online computation is therefore an $r\times r$ dense solve at each time step, not a solve with the ill-conditioned raw mass matrix.  Recomputing $\mathbf T_r$ during the time march would define time-dependent coordinates and introduce additional $\dot {\mathbf T}_r$ terms; that is not the method analyzed or implemented here.  At every time level, the reconstruction remains in the same retained space $V_{N,r}$; the following stability and error analysis concerns precisely this fixed-feature space evolution.

\section{Stability and error analysis}\label{sec:analysis}

Classical analyses of Galerkin approximations for parabolic PDEs reduce the error analysis to the stability of the projected evolution and the approximation properties of the trial space \cite{douglas1970galerkin,thomee2006galerkin}. In the present Evo-GTransNet method, however, the mass, stiffness, and load forms are assembled by quadrature, so their comparison with the corresponding continuous forms introduces an additional variational-crime, or consistency, error \cite{strang1972variational}.
The numerical analysis is carried out on the fixed retained space selected by the quadrature-weighted truncated SVD. On this space, the identity-mass relation holds exactly with respect to the quadrature inner product $(\cdot,\cdot)_q$ defined above. 
The consistency residual $\eta_q$ introduced below compares the continuous weak formulation with the quadrature forms on this fixed retained space; it does not compare the quadrature-selected retained space with a separately constructed exact-mass retained space. 


\subsection{Energy stability}

Let $a(\bx,t)\ge a_0>0$ and $f=0$ in
\eqref{eq:linear_parabolic}. Then $\widetilde{\vb}_q=0$, and \eqref{eq:retained_ode} becomes 
\begin{equation}
  \dot{\valpha}(t)+\widetilde{\mathbf K}_q(t)\valpha(t)=\mathbf 0,
\end{equation}
with
\begin{equation*}
  (\widetilde{\mathbf K}_q)_{ij}(t)
  =\sum_{m=1}^Qw_ma(\bx_m,t)
\nabla\widetilde\varphi_j(\bx_m)\cdot\nabla\widetilde\varphi_i(\bx_m).
\end{equation*}
\begin{theorem}[Quadrature-semidiscrete energy stability]\label{thm:energy}
Assume that $\widetilde{\mathbf K}_q(t)=\widetilde{\mathbf K}_q(t)^\top\succeq0$ for almost every $t\in(0,T)$. Then, for almost every $t\in(0,T)$, the retained Evo-GTransNet solution satisfies
\begin{equation}
  \frac12\frac{\dd}{\dd t}\norm{u_{N,r}(t)}_q^2
  +\valpha(t)^\top\widetilde{\mathbf K}_q(t)\valpha(t)=0.
  \label{eq:energy_identity}
\end{equation}
Moreover, for almost every $t\in(0,T)$, 
\begin{equation}
  \valpha(t)^\top\widetilde{\mathbf K}_q(t)\valpha(t)
  =\sum_{m=1}^Qw_ma(\bx_m,t)
   |\nabla u_{N,r}(\bx_m,t)|^2.
   \label{eq:quadrature_diffusion_energy}
\end{equation}
In particular, $\norm{u_{N,r}(t)}_q\le\norm{u_{N,r}(0)}_q$ for every $t\in[0,T]$.
\end{theorem}
\begin{proof}
For the absolutely continuous coefficient path, multiplying the normalized coefficient equation by $\valpha(t)^\top$ for almost every $t\in(0,T)$ gives
\begin{equation}
  \valpha^\top\dot{\valpha}
  +\valpha^\top\widetilde{\mathbf K}_q(t)\valpha=0.
\end{equation}
Because the retained basis is fixed in time and has identity mass with respect to the quadrature inner product, we have
$$
\valpha^\top\dot{\valpha}
=\frac12\frac{\dd}{\dd t}\norm{u_{N,r}(t)}_q^2
$$
for almost every $t\in(0,T)$, which proves \eqref{eq:energy_identity}. For the homogeneous formulation,
$\nabla u_{N,r}=\sum_{i=1}^r\alpha_i\nabla\widetilde\varphi_i$.
Substituting this expression into the definition of $\widetilde{\mathbf K}_q(t)$ gives the second identity \eqref{eq:quadrature_diffusion_energy}. Finally, since $\widetilde{\mathbf K}_q(t)\succeq0$, we have
$\valpha(t)^\top\widetilde{\mathbf K}_q(t)\valpha(t)\ge0$.
Integrating \eqref{eq:energy_identity} from $0$ to $t$ yields
$$
  \frac12\norm{u_{N,r}(t)}_q^2
  +\int_0^t\valpha(s)^\top\widetilde{\mathbf K}_q(s)\valpha(s)\,\dd s
  =\frac12\norm{u_{N,r}(0)}_q^2.
$$
Since the integral is nonnegative, it follows that
$\norm{u_{N,r}(t)}_q^2\le\norm{u_{N,r}(0)}_q^2$.
Taking square roots completes the proof.
\end{proof}

\begin{remark}
If the construction is instead repeated with the exact continuous mass inner product, using a separately defined exact-mass-normalized basis, the identical coefficient algebra gives the corresponding continuous $L^2$ energy identity.  The quadrature-selected basis used in the implementation is not assumed to be continuously $L^2$-orthonormal.
\end{remark}

\begin{theorem}[Implicit midpoint contractivity]\label{thm:im}
If the midpoint stiffness matrix satisfies
\[
\widetilde{\mathbf K}_q^{n+1/2}
=(\widetilde{\mathbf K}_q^{n+1/2})^\top\succeq0,
\]
the homogeneous implicit midpoint update
\begin{equation}
  \left(\mathbf I_r+\frac{\Delta t}{2}\widetilde{\mathbf K}_q^{n+1/2}\right)\valpha^{n+1}
  =\left(\mathbf I_r-\frac{\Delta t}{2}\widetilde{\mathbf K}_q^{n+1/2}\right)\valpha^n
\end{equation}
obeys $\norm{\valpha^{n+1}}_2\le\norm{\valpha^n}_2$ for every $\Delta t>0$.  Because $\norm{u_{N,r}^n}_q=\norm{\valpha^n}_2$, it is equivalently contractive in the assembly-quadrature norm of the retained function.  
\end{theorem}
\begin{proof}
Since $\widetilde{\mathbf K}_q^{n+1/2}\succeq0$, $\mathbf I_r+(\Delta t/2)\widetilde{\mathbf K}_q^{n+1/2}$ is symmetric positive definite, so the update is well defined.
Rearrange the update as
\begin{equation}
  \valpha^{n+1}-\valpha^n
  +\frac{\Delta t}{2}\widetilde{\mathbf K}_q^{n+1/2}
   (\valpha^{n+1}+\valpha^n)=\mathbf 0.
\end{equation}
Taking the Euclidean inner product with $\valpha^{n+1}+\valpha^n$ gives the
discrete energy identity
\begin{equation}
  \norm{\valpha^{n+1}}_2^2-\norm{\valpha^n}_2^2
  +\frac{\Delta t}{2}
   (\valpha^{n+1}+\valpha^n)^\top
   \widetilde{\mathbf K}_q^{n+1/2}
   (\valpha^{n+1}+\valpha^n)=0.
\end{equation}
The last term is nonnegative because $\widetilde{\mathbf K}_q^{n+1/2}$ is symmetric positive
semidefinite.  Hence $\norm{\valpha^{n+1}}_2^2\le\norm{\valpha^n}_2^2$.
The argument is stepwise, so $\widetilde{\mathbf K}_q^{n+1/2}$ may change from one time step to
the next.
\end{proof}
\begin{remark}
    The norm equivalence above relies specifically on the assembly normalization $\widetilde{\bm\Phi}^\top\mathbf W_q\widetilde{\bm\Phi}=\mathbf I_r$.  A separate validation rule or the continuous $L^2$ inner product generally induces a different retained mass matrix; consequently, this theorem makes no contractivity assertion in either of those norms.
\end{remark}
\subsection{Error decomposition and comparison defects}\label{sec:error_decomposition}

We now return to the general forced problem, retaining the domain,
quadrature rule, and fixed trial spaces defined above. For the
remainder of the error analysis, assume that
$
  u\in C([0,T];L^2(\Omega)).
$
Assume also that all functions and derivatives appearing in the quadrature expressions below---including $u$, $u_t$, $\nabla u$, $f$, and the comparison functions introduced below together with their time derivatives---admit well-defined values at the quadrature nodes, with the time measurability required by the displayed integrals. Assume further that each nodal coefficient function $a(\bx_m,\cdot)$ is measurable and that
$
  \widetilde{\vb}_q\in L^1(0,T;\mathbb R^r).
$
Together with the coefficient and stability assumptions below, these conditions give an absolutely continuous semidiscrete coefficient path and justify the energy integrations.

Define the discrete $H^1$ norm and the quadrature diffusion form by
\begin{align*}
 & 
  \norm{v}_{1,q}^2
  :=\norm{v}_q^2+\sum_{\ell=1}^d\norm{\partial_{x_\ell}v}_q^2,
  \\
 & \mathcal A_q(t;v,z)
  :=\sum_{m=1}^Qw_ma(\bx_m,t)
      \nabla v(\bx_m)\cdot\nabla z(\bx_m).
\end{align*}
On $V_{N,r}$, $\|\cdot\|_{1,q}$ is a norm because it dominates $\|\cdot\|_q$, which is itself a norm by Proposition~\ref{prop:mass}. 
For a quantity $g$ evaluated at the quadrature nodes, we define the retained-space discrete dual seminorm by
\begin{equation}
  \norm{g}_{-1,q,r}
  :=\sup_{0\ne z\in V_{N,r}}
  \frac{|(g,z)_q|}{\norm{z}_{1,q}}.
  \label{eq:discrete_dual_norm}
\end{equation}
Let $P_{r,q}$ denote the quadrature-orthogonal projection onto $V_{N,r}$,
$$
  (P_{r,q}g,z)_q=(g,z)_q,
  \qquad \forall\,z\in V_{N,r}.
$$
This projection is well defined because $(\cdot,\cdot)_q$ is an inner product on $V_{N,r}$ by Proposition~\ref{prop:mass}.  The quadrature rule and retained space are fixed, so $P_{r,q}$ is independent of time.  Since the semi-discrete solution $u_{N,r}(t)\in V_{N,r}$ is the function reconstructed from the coefficient solution of \eqref{eq:retained_ode}, it satisfies the following  time-continuous quadrature-assembled semidiscrete Galerkin problem:
\begin{equation}
  (\partial_t u_{N,r},z)_q
  +\mathcal A_q(t;u_{N,r},z)
  =(f,z)_q,
  \qquad \forall z\in V_{N,r},
  \label{eq:q_galerkin}
\end{equation}
with $u_{N,r}(0)=P_{r,q}u_0$.
(By the identity-mass relation, the projected initial condition is equivalent to the coefficient initialization specified in Section~\ref{sec:transformed_system}). Let $u_{N,r}^n$ denote the corresponding fully discrete solution to the implicit midpoint approximation of this time-continuous semidiscrete solution, equivalently the one produced from \eqref{eq:im_retained}. To distinguish the approximation error in the full sampled space from the error introduced by the reduction to the retained space, let us choose a $v_N\in C^1([0,T];V_N)$ and set
$$
  v_{N,r}(t):=P_{r,q}v_N(t)\in V_{N,r}.
$$
Then $v_{N,r}\in C^1([0,T];V_{N,r})$ and
$\partial_t v_{N,r}=P_{r,q}\partial_t v_N$.  At every time level, it holds
\begin{equation}
\begin{aligned}
  u(t_n)-u_{N,r}^n
  ={}&\underbrace{u(t_n)-v_N(t_n)}_{\text{full-space approximation}}
  +\underbrace{v_N(t_n)-v_{N,r}(t_n)}_{\text{retained-space reduction}}
  \\[1mm]
  &+\underbrace{v_{N,r}(t_n)-u_{N,r}(t_n)}_{\text{Galerkin evolution and quadrature}}
  +\underbrace{u_{N,r}(t_n)-u_{N,r}^n}_{\text{time discretization}}.
\end{aligned}
\label{eq:error_split_retained}
\end{equation}

\paragraph{Comparison defects}
Let
$
  \eta_N^{\rm full}:[0,T]\to[0,\infty), 
  ~
  \eta_{t,N}^{\rm full}:(0,T)\to[0,\infty)
$
be comparison functions satisfying
\begin{align}
  \norm{u(t)-v_N(t)}_{L^2}
  +\norm{u(t)-v_N(t)}_{1,q}
  &\le\eta_N^{\rm full}(t),
  &&0\le t\le T,
  \label{eq:eta_full}\\
  \norm{u_t(t)-\partial_t v_N(t)}_{-1,q,r}
  &\le\eta_{t,N}^{\rm full}(t),
  &&\text{for almost every }t\in(0,T).
  \label{eq:eta_full_time}
\end{align}
Assume that $\eta_{t,N}^{\rm full}\in L^2(0,T). $
Define the retained-space reduction defect by
\begin{equation}
  \eta_{\rm red}(t)
  :=\norm{v_N(t)-v_{N,r}(t)}_{L^2}
    +\norm{v_N(t)-v_{N,r}(t)}_{1,q}.
  \label{eq:eta_reduction}
\end{equation}
Consequently,
\begin{align}
  \norm{u(t)-v_{N,r}(t)}_{L^2}
  +\norm{u(t)-v_{N,r}(t)}_{1,q}
  &\le\eta_N^{\rm full}(t)+\eta_{\rm red}(t),
  \label{eq:retained_spatial_defect}\\
  \norm{u_t(t)-\partial_t v_{N,r}(t)}_{-1,q,r}
  &=\norm{u_t(t)-\partial_t v_N(t)}_{-1,q,r}
  \le\eta_{t,N}^{\rm full}(t).
  \label{eq:retained_time_defect}
\end{align}
The equality in \eqref{eq:retained_time_defect} follows from the $q$-orthogonality of $P_{r,q}$ on the retained test space.  Hence no separate temporal reduction defect is needed.  The defect \eqref{eq:eta_reduction} measures the loss under the specified projection onto the SVD-selected space, and  without additional norm-specific estimates, it is not determined solely by the discarded singular values.

The initial projection mismatch is
\begin{equation}
  \eta_0:=\norm{v_{N,r}(0)-P_{r,q}u_0}_q.
  \label{eq:initial_comparison_defect}
\end{equation}
For the projected comparison path,
\begin{equation}
  \eta_0
  =\norm{P_{r,q}(v_N(0)-u_0)}_q
  \le\norm{v_N(0)-u_0}_q
  \le\eta_N^{\rm full}(0).
  \label{eq:initial_defect_control}
\end{equation}
Thus $\eta_0$ identifies the effect of projected initialization, although it is controlled by the full-space comparison defect.

\paragraph{Fixed-feature space stability and quadrature consistency}
For the fixed sampled dictionary, retained space, and quadrature rule, suppose that there are constants $c_{0,q}>0$ and $c_{\mathcal A}>0$ such that
\begin{align}
  c_{0,q}\norm{z}_{L^2}^2
  \le\norm{z}_q^2,
  &\qquad\forall\,z\in V_{N,r},
  \label{eq:mass_lower_bound}\\
  \mathcal A_q(t;z,z)
  \ge c_{\mathcal A}\norm{z}_{1,q}^2,
  &\qquad\forall\, z\in V_{N,r},\;0\le t\le T.
  \label{eq:discrete_coercivity}
\end{align}
The upper bound on $a(\bx,t)$ gives
\begin{equation}
  |\mathcal A_q(t;v,z)|
  \le a_1\norm{v}_{1,q}\norm{z}_{1,q}
  \label{eq:discrete_boundedness}
\end{equation}
for all functions whose nodal gradients are defined.  For each fixed finite-dimensional retained space, \eqref{eq:mass_lower_bound} holds for some $c_{0,q}>0$ by equivalence of the continuous $L^2$ norm and the quadrature norm on $V_{N,r}$; no lower bound uniform under refinement is asserted.  In contrast, \eqref{eq:discrete_coercivity} is an additional retained-space stability assumption. Since $\mathcal A_q$ contains only gradients, it does not follow from $a(\bx,t)\ge a_0$ alone without a corresponding discrete Poincar\'e property.

Define the quadrature consistency functional, its retained dual norm, and its accumulated size by
\begin{align}
  \mathcal R_q(t;z)
  &:= (u_t,z)_q+\mathcal A_q(t;u,z)-(f,z)_q,
  \qquad \forall\,z\in V_{N,r},
  \notag\\
  R_q(t)
  &:=\sup_{0\ne z\in V_{N,r}}
  \frac{|\mathcal R_q(t;z)|}{\norm{z}_{1,q}},\quad
  \eta_q
  :=\left[\int_0^T R_q(t)^2\dd t\right]^{1/2}.
  \label{eq:quadrature_residual}
\end{align}
For almost every $t\in(0,T)$ and every $z\in V_{N,r}$, the continuous weak equation gives
\begin{equation}
\begin{aligned}
  \mathcal R_q(t;z)
  ={}&\bigl[(u_t,z)_q-\langle u_t,z\rangle\bigr]
  +\left[\mathcal A_q(t;u,z)
    -\int_\Omega a(\bx,t)\nabla u\cdot\nabla z\dd\bx\right]
  \\
  &+\left[\int_\Omega f(\bx,t)z(\bx)\dd\bx-(f,z)_q\right].
\end{aligned}
\end{equation}
Thus $R_q$ measures the combined quadrature inconsistency in the time-derivative, diffusion, and load terms on the retained test space and we assume that $R_q\in L^2(0,T)$.
For compactness, set
\begin{equation}
  \mathcal E_t
  :=\left[\int_0^T
     \bigl(\eta_{t,N}^{\rm full}(t)\bigr)^2\dd t\right]^{1/2},\qquad
      \mathcal E_{\rm sp}
  :=\sup_{0\le t\le T}
     \bigl(\eta_N^{\rm full}(t)+\eta_{\rm red}(t)\bigr).
  \label{eq:comparison_defect_summary}
\end{equation}
The assumptions $\eta_{t,N}^{\rm full}\in L^2(0,T)$ and $R_q\in L^2(0,T)$ imply $\mathcal E_t<\infty$ and $\eta_q<\infty$, respectively. We additionally assume that $\mathcal E_{\rm sp}<\infty.$

\subsection{Conditional semidiscrete comparison estimate and fully discrete error estimate}\label{sec:semidiscrete_error}

\begin{lemma}[Conditional semidiscrete comparison estimate]\label{lem:semidiscrete_error}
Under the preceding assumptions, there exists a constant $C_{\rm sd}>0$, depending only on $T$, $a_1$, $c_{\mathcal A}$, $c_{0,q}$, and the retained space $V_{N,r}$, such that
\begin{equation}
  \sup_{0\le t\le T}\norm{u(t)-u_{N,r}(t)}_{L^2}
  \le C_{\rm sd}
  \bigl(\mathcal E_{\rm sp}+\mathcal E_t+\eta_q\bigr).
  \label{eq:semidiscrete_error_bound}
\end{equation}
\end{lemma}
\begin{proof}
Set $e_r(t)=v_{N,r}(t)-u_{N,r}(t)$. For every $z\in V_{N,r}$ and for almost every $t\in(0,T)$, the $q$-orthogonality of $P_{r,q}$ implies
$$
  (\partial_t v_{N,r}-u_t,z)_q
  =(P_{r,q}\partial_t v_N-u_t,z)_q
  =(\partial_t v_N-u_t,z)_q.
$$
Combining \eqref{eq:q_galerkin} with the definition of $\mathcal R_q$ gives, for every $z\in V_{N,r}$ and for almost every $t\in(0,T)$,
\begin{align}
  (\partial_t e_r,z)_q+\mathcal A_q(t;e_r,z)
  &=(\partial_t v_N-u_t,z)_q
    +\mathcal A_q(t;v_{N,r}-u,z)
  +\mathcal R_q(t;z).
  \label{eq:semidiscrete_error_identity}
\end{align}
By \eqref{eq:quadrature_residual}, we have
$|\mathcal R_q(t;z)|\le R_q(t)\|z\|_{1,q}$ for every $z\in V_{N,r}$.
Note that $e_r$ is absolutely continuous. Taking $z=e_r(t)$ and using \eqref{eq:discrete_coercivity}, \eqref{eq:discrete_boundedness}, \eqref{eq:retained_spatial_defect}, and Young's inequality yields, for almost every $t\in(0,T)$,
\begin{align}
  \frac12\frac{\dd}{\dd t}\norm{e_r(t)}_q^2
  +\frac{c_{\mathcal A}}{2}\norm{e_r(t)}_{1,q}^2
  &\le C\Bigl[
     \bigl(\eta_{t,N}^{\rm full}(t)\bigr)^2
     +\bigl(\eta_N^{\rm full}(t)+\eta_{\rm red}(t)\bigr)^2
  +R_q(t)^2\Bigr].
  \label{eq:semidiscrete_energy_bound}
\end{align}
Here $e_r(0)=P_{r,q}(v_N(0)-u_0)$, so
$\norm{e_r(0)}_q=\eta_0$. Integration of \eqref{eq:semidiscrete_energy_bound} gives
\begin{equation}
  \sup_{0\le t\le T}\norm{e_r(t)}_q
  \le C\bigl(\eta_0+\mathcal E_{\rm sp}+\mathcal E_t+\eta_q\bigr).
  \label{eq:galerkin_error_q}
\end{equation}
Furthermore
$$
  \norm{u(t)-u_{N,r}(t)}_{L^2}
  \le\norm{u(t)-v_{N,r}(t)}_{L^2}
    +c_{0,q}^{-1/2}\norm{e_r(t)}_q.
$$
Combining this inequality with \eqref{eq:retained_spatial_defect}, \eqref{eq:initial_defect_control} and \eqref{eq:galerkin_error_q} proves \eqref{eq:semidiscrete_error_bound}.
\end{proof}



\begin{lemma}[Implicit midpoint consistency]\label{lem:im_consistency}
Assume that $\valpha\in C^3([0,T];\mathbb R^r)$ solves
$$
 \dot{\valpha}(t)+\widetilde{\mathbf K}_q(t)\valpha(t)
 =\widetilde{\vb}_q(t).
$$
Define the one-step residual
\begin{align}
  \boldsymbol\rho_{\rm IM}^{n+1/2}
  &:={\valpha}(t_{n+1})-{\valpha}(t_n)
  +\frac{\Delta t}{2}\widetilde{\mathbf K}_q(t_{n+1/2})
   \bigl({\valpha}(t_{n+1})+{\valpha}(t_n)\bigr)\\
  &\quad-\Delta t\,\widetilde{\vb}_q(t_{n+1/2}).
  \label{eq:im_residual_definition}
\end{align}
Then
\begin{equation}
  \norm{\boldsymbol\rho_{\rm IM}^{n+1/2}}_2
  \le\Delta t^3\Theta_{\rm IM},
  \label{eq:im_residual_bound}
\end{equation}
where one admissible choice, independent of $\Delta t$, is
\begin{equation}
  \Theta_{\rm IM}
  :=\frac{1}{24}\norm{\valpha^{(3)}}_{L^\infty(0,T;\ell^2)}
   +\frac{1}{8}\norm{\widetilde{\mathbf K}_q}_{L^\infty(0,T;2)}
    \norm{\ddot{\valpha}}_{L^\infty(0,T;\ell^2)}.
  \label{eq:im_residual_explicit}
\end{equation}
Here $\valpha^{(j)}:=\mathrm d^j\valpha/\mathrm dt^j$, and
\begin{align*}
  \|\widetilde{\mathbf K}_q\|_{L^\infty(0,T;2)}
  &:=\operatorname*{ess\,sup}_{0\le t\le T}\|\widetilde{\mathbf K}_q(t)\|_2,\qquad
  \|\valpha^{(j)}\|_{L^\infty(0,T;\ell^2)}
  :=\operatorname*{ess\,sup}_{0\le t\le T}\|\valpha^{(j)}(t)\|_2.
\end{align*}
\end{lemma}
\begin{proof}
Evaluating the semidiscrete coefficient equation at
$\bar t_n=t_{n+1/2}$ gives
$$
  \widetilde{\vb}_q(\bar t_n)
  =
  \dot{\valpha}(\bar t_n)
  +\widetilde{\mathbf K}_q(\bar t_n)
   \overline{\valpha}_n,
  \qquad
  \overline{\valpha}_n:=\valpha(\bar t_n).
$$
Set
$
  \valpha_\pm
  :=\valpha\left(\bar t_n\pm\frac{\Delta t}{2}\right).
$
Substituting the midpoint identity into \eqref{eq:im_residual_definition} and collecting the two stiffness terms yields 
\begin{equation}
  \boldsymbol\rho_{\rm IM}^{n+1/2}
  =\bigl(\valpha_+-\valpha_--\Delta t\,\dot{\valpha}(\bar t_n)\bigr)
  +\Delta t\,\widetilde{\mathbf K}_q(\bar t_n)
   \left(\frac{\valpha_++\valpha_-}{2}-\overline{\valpha}_n\right).
  \label{eq:im_residual_identity}
\end{equation}
The midpoint Taylor estimates
$$
  \norm{\valpha_+-\valpha_--\Delta t\,\dot{\valpha}(\bar t_n)}_2
  \le\frac{\Delta t^3}{24}
       \norm{\valpha^{(3)}}_{L^\infty(0,T;\ell^2)}
$$
and
$$
  \left\|\frac{\valpha_++\valpha_-}{2}-\overline{\valpha}_n\right\|_2
  \le\frac{\Delta t^2}{8}
       \norm{\ddot{\valpha}}_{L^\infty(0,T;\ell^2)}
$$
imply \eqref{eq:im_residual_bound}.
\end{proof}
\begin{remark}
A sufficient regularity condition is
$\widetilde{\mathbf K}_q,\widetilde{\vb}_q\in C^2([0,T])$, which will ensure
$\valpha\in C^3([0,T];\mathbb R^r)$ for this finite-dimensional system.    
\end{remark}

\begin{theorem}[Conditional fully discrete error estimate]\label{thm:full_error}
Suppose the fixed-space stability, approximation, and quadrature consistency conditions of Lemma~\ref{lem:semidiscrete_error} hold and the semidiscrete coefficient solution satisfies
$
  \valpha\in C^3([0,T];\mathbb R^r).
$
Then
\begin{equation}
  \max_{0\le n\le N_t}\norm{u(t_n)-u_{N,r}^n}_{L^2}
  \le C_{\rm fd}\bigl(
       \mathcal E_{\rm sp}+\mathcal E_t+\eta_q
       +\Delta t^2\Theta_{\rm IM}\bigr),
  \label{eq:err_bound}
\end{equation}
where $C_{\rm fd}>0$ is a constant independent of $\Delta t$.  
\end{theorem}
\begin{proof}
Lemma~\ref{lem:semidiscrete_error} controls $u-u_{N,r}$. For the temporal error, set
$$
  \vd^n:=\valpha(t_n)-\valpha^n.
$$
Both coefficient paths use the same projected initial data, so $\vd^0=\mathbf 0$. Subtracting the numerical update from \eqref{eq:im_residual_definition} gives
\begin{equation}
  \left(\mathbf I_r+\frac{\Delta t}{2}\widetilde{\mathbf K}_q^{n+1/2}\right)\vd^{n+1}
  =\left(\mathbf I_r-\frac{\Delta t}{2}\widetilde{\mathbf K}_q^{n+1/2}\right)\vd^n
   +\boldsymbol\rho_{\rm IM}^{n+1/2}.
  \label{eq:im_error_recursion}
\end{equation}
Set $s=\Delta t/2$, $\mathbf K_n=\widetilde{\mathbf K}_q^{n+1/2}$, and
$$
  \mathbf C_n
  :=(\mathbf I_r+s\mathbf K_n)^{-1}(\mathbf I_r-s\mathbf K_n).
$$
Then \eqref{eq:im_error_recursion} is equivalent to
$$
  \vd^{n+1}
  =\mathbf C_n\vd^n
   +(\mathbf I_r+s\mathbf K_n)^{-1}\boldsymbol\rho_{\rm IM}^{n+1/2}.
$$
By the representation \eqref{eq:raw_to_normalized}, the positivity of the quadrature weights, and $a(\bx_m,t)\ge a_0>0$, $\mathbf K_n=\widetilde{\mathbf K}_q^{n+1/2}$ is symmetric positive semidefinite. Hence the spectral theorem gives 
$$
  \norm{\mathbf C_n}_2\le1,
  \qquad
  \norm{(\mathbf I_r+s\mathbf K_n)^{-1}}_2\le1.
$$
Consequently,
\begin{equation}
  \norm{\vd^{n+1}}_2
  \le\norm{\vd^n}_2+\norm{\boldsymbol\rho_{\rm IM}^{n+1/2}}_2.
\end{equation}
Since $\vd^0=\mathbf0$, iteration of this inequality and Lemma~\ref{lem:im_consistency} give 
$$
\begin{aligned}
  \norm{\vd^n}_2
  &\le
  \sum_{j=0}^{n-1}
  \norm{\boldsymbol\rho_{\rm IM}^{j+1/2}}_2
  \le
  n\Delta t^3\Theta_{\rm IM}
  \le
  T\Delta t^2\Theta_{\rm IM}.
\end{aligned}
$$
Therefore,
$$
  \max_{0\le n\le N_t}
  \norm{\vd^n}_2
  \le
  T\Delta t^2\Theta_{\rm IM}.
$$
The quadrature identity-mass relation gives
$$
  \norm{u_{N,r}(t_n)-u_{N,r}^n}_q
  =\norm{\vd^n}_2.
$$
By \eqref{eq:mass_lower_bound},
$$
  \norm{u_{N,r}(t_n)-u_{N,r}^n}_{L^2}
  \le
  c_{0,q}^{-1/2}\norm{\vd^n}_2
  \le
  c_{0,q}^{-1/2}
  T\Delta t^2\Theta_{\rm IM}.
$$
Combining this temporal estimate with Lemma~\ref{lem:semidiscrete_error} and the triangle inequality proves \eqref{eq:err_bound}.
\end{proof}

\begin{remark}
Theorem~\ref{thm:full_error} is a conditional estimate for each fixed sampled dictionary and quadrature rule, not a probabilistic approximation theorem for GTransNet features. The retained-space reduction is the fixed $q$-orthogonal projection loss defined in \eqref{eq:eta_reduction}. A spatial-refinement convergence rate additionally requires the fixed-space stability constants to remain bounded and the stated comparison and quadrature defects to tend to zero.
\end{remark}

\section{Numerical experiments}\label{sec:numerics}

Unless explicitly stated otherwise, the manufactured parabolic benchmarks in this section use homogeneous Dirichlet boundary conditions. The boundary condition is imposed strongly through the boundary factor $D$ in the trial functions; it is neither fitted by least squares nor enforced through a boundary penalty. The manufactured exact solutions satisfy this boundary condition. In the terminology of verification and validation, these are verification tests based on the method of manufactured solutions rather than physical validation cases \cite{roache2002mms}.
The source term in each manufactured linear benchmark is obtained by substituting the stated exact solution into the model diffusion problem \eqref{eq:linear_parabolic}.
unless otherwise stated. 
When the diffusion coefficient $a$ depends on time, both $a$ and $f$ are evaluated at the temporal midpoint $t_{n+1/2}$ used in the implicit midpoint scheme.

\subsection{Experimental protocol and overview}

The numerical comparisons use four prescribed fixed-feature dictionaries.
\begin{enumerate}[leftmargin=2em]
\item \textbf{TransNet}: \eqref{eq:trans_feature}, with independently sampled unit-sphere directions and $\xi_j\stackrel{\rm iid}{\sim}U[0,1]$ in its single hidden layer.
\item \textbf{GTransNet}: \eqref{eq:gtrans_first}--\eqref{eq:variance_control}, with one additional hidden layer and $\delta=0.5$.
\item \textbf{Gaussian RF}: $\psi_j(\bx)=\tanh\{\gamma[\ba_j^\top(\bx-\bx_c)+R\xi_j]\}$, with all components of $\ba_j$ and $\xi_j$ sampled independently from $\mathcal N(0,1)$.
\item \textbf{Uniform ELM}: the same one-layer formula, with all components of $\ba_j$ and $\xi_j$ sampled independently from $U[-1,1]$.
\end{enumerate}

Random sampling uses NumPy's \texttt{default\_rng}, initialized with the stated integer seed. For GTransNet, each seed generates the first-layer directions, offsets, and second-layer weights in that order. The one- and two-dimensional dictionaries use $\bx_c=(1/2,\ldots,1/2)$ and $R=0.8$, while the three-dimensional dictionary uses the same center and $R=0.9$.

Within each comparison, all dictionaries use the same output-evolution solver, boundary factor, time step, and quadrature-mass orthonormalization. Comparisons are performed at the same evolved output dimension $N$. GTransNet additionally uses a first-layer width $N_1$ to generate its final $N$ features, so the reported error comparisons are not cost matched. For the two one-dimensional benchmark families, $\gamma$ is selected by minimizing the mean validation error over the pilot seeds $100$--$104$ and is then fixed for the disjoint evaluation seeds $0$--$19$. No parameter is tuned on the evaluation seeds. The candidate sets, selection protocol, and all selected or prescribed values of $\gamma$ are reported in \ref{app:gamma_candidates}.

The one-dimensional Galerkin systems use positive Gauss--Legendre quadrature for assembly and separate composite-midpoint grids for validation. The two- and three-dimensional experiments use distinct tensor-product Gauss--Legendre rules for assembly and validation, so the reported errors are evaluated away from the assembly nodes. The feature-derivative formulas are given in \ref{app:feature_derivatives}, the assembly conventions in \ref{app:assembly}, and the multidimensional cross-grid assessments in \ref{app:cross_grid}.

Unless stated otherwise, the retained linear systems are advanced by the implicit midpoint rule with truncation threshold $\tau=10^{-12}$. Table~\ref{tab:experiment_overview} summarizes the principal configurations. The sensitivity studies use the corresponding one-dimensional base configuration except for the parameter being varied. Here $Q_a$ and $Q_v$ denote the total numbers of assembly and validation nodes, respectively. The entry $N/N_1$ records the evolved dimension and the GTransNet first-layer width, respectively.

\begin{table}[ht!]
\centering
\caption{Overview of the principal numerical configurations.  A dash indicates that validation is performed against a coefficient-space reference rather than on a separate spatial grid.}
\label{tab:experiment_overview}
\tableformat
\begin{tabular}{@{}p{0.45\textwidth}cccc@{}}
\toprule
Benchmark & $N/N_1$ & $Q_a/Q_v$ & $T/\Delta t$ & seeds\\
\midrule
1D global space--time diagnostic & $60/240$ & $320/1000$ & $0.25/0.0025$ & 3\\
1D oscillatory coefficient & $60/240$ & $320/4096$ & $0.25/0.0025$ & 20\\
1D $\eps$-dependent target & $180/720$ & $512/4096$ & $0.25/0.0025$ & 20\\
Implicit-midpoint convergence & $60/240$ & $320/\text{--}$ & varied & 1\\
2D fixed/time-dependent operator & $150/600$ & $36^2/64^2$ & $0.06/0.003$ & 10\\
2D nonseparable evolution & $180/720$ & $42^2/64^2$ & $0.16/0.004$ & 10\\
3D nonseparable evolution & $320/1280$ & $20^3/22^3$ & $0.12/0.006$ & 5\\
2D Allen--Cahn & $150/600$ & $40^2/60^2$ & $1.0/0.002$ & 5\\
\bottomrule
\end{tabular}
\end{table}

For validation nodes $\bx_q^{(v)}$ and weights $w_q^{(v)}$, the reported final relative error is
\begin{equation}
  e_{\rm rel}(T)=
  \left[
    \frac{\sum_q w_q^{(v)}|u_N(\bx_q^{(v)},T)-u_{\rm ref}(\bx_q^{(v)},T)|^2}
         {\sum_q w_q^{(v)}|u_{\rm ref}(\bx_q^{(v)},T)|^2}
  \right]^{1/2}.
  \label{eq:validation_error}
\end{equation}
Initial projection errors use the same formula at $t=0$.  Unless explicitly stated otherwise, displayed error bars are sample standard deviations over the listed seeds, computed with divisor $n-1$.  In the compact statistical tables below, each two-line entry places the sample mean on the first line and the sample standard deviation in parentheses on the second.  Displayed means and standard deviations are rounded to three significant digits; calculations and improvement factors use the unrounded data.  The column headed ``gain T/G'' is $\bar e_{\rm TransNet}/\bar e_{\rm GTransNet}$.

\subsection{One-dimensional approximation benchmarks}

\subsubsection{An oscillatory-coefficient parabolic benchmark}

We solve the equation
\begin{equation}
  u_t-(a_\eps(x)u_x)_x=f(x,t),\qquad
  a_\eps(x)=1+\frac12\cos\left(\frac{2\pi x}{\eps}\right),
  \label{eq:1d_multi}
\end{equation}
with exact solution
\begin{equation}
  u(x,t)=e^{-t}\sin(16\pi x).
\end{equation}
The exact spatial profile is fixed as $\eps$ varies, whereas the diffusion coefficient and manufactured forcing vary with $\eps$.  This is therefore an oscillatory-coefficient, high-frequency manufactured test rather than a homogenization-style test with an $\eps$-dependent exact solution.  We use the configuration in Table~\ref{tab:experiment_overview}; Table~\ref{tab:main1d} reports statistics on the twenty evaluation seeds after selecting $\gamma$ on the disjoint pilot set.

\begin{table}[ht!]
\centering
\caption{One-dimensional oscillatory-coefficient benchmark.  Entries are twenty-seed final relative $L^2$ errors; the second line in each entry is the sample standard deviation.}
\label{tab:main1d}
\tableformat
\begin{tabular}{@{}cccccc@{}}
\toprule
$\eps$ & TransNet & GTransNet & Gaussian RF & Uniform ELM & gain T/G\\
\midrule
0.20 & \meanstd{4.50\times10^{-2}}{5.43\times10^{-2}} & \meanstd{3.60\times10^{-4}}{1.52\times10^{-4}} & \meanstd{4.48\times10^{-1}}{3.00\times10^{-1}} & \meanstd{3.37\times10^{-1}}{2.76\times10^{-1}} & $125.0$\\
0.10 & \meanstd{8.93\times10^{-2}}{1.30\times10^{-1}} & \meanstd{4.36\times10^{-4}}{2.10\times10^{-4}} & \meanstd{6.81\times10^{-1}}{3.19\times10^{-1}} & \meanstd{5.42\times10^{-1}}{4.16\times10^{-1}} & $204.8$\\
0.05 & \meanstd{4.27\times10^{-2}}{5.03\times10^{-2}} & \meanstd{7.74\times10^{-4}}{4.31\times10^{-4}} & \meanstd{3.44\times10^{-1}}{2.11\times10^{-1}} & \meanstd{2.53\times10^{-1}}{1.69\times10^{-1}} & $55.2$\\
\bottomrule
\end{tabular}
\end{table}

Because the exact spatial profile is fixed, this experiment measures representation of a prescribed high-frequency solution under a common oscillatory-coefficient Galerkin operator.  At the stated hyperparameters and matched nominal $N$, GTransNet has the lowest mean error.  The TransNet distribution is visibly seed dependent, so we also compare medians.  For $\eps=0.20,0.10,0.05$, its median errors are $3.17\times10^{-2}$, $5.35\times10^{-2}$, and $3.08\times10^{-2}$, whereas the corresponding GTransNet medians are $3.83\times10^{-4}$, $4.40\times10^{-4}$, and $7.06\times10^{-4}$.  The median comparison supports the same ordering as the means.

\subsubsection{An \texorpdfstring{$\eps$}{epsilon}-dependent exact-solution benchmark}

To test a more demanding situation in which the solution itself depends on the small scale, we solve the equation
\begin{equation}
  u_t-\nu u_{xx}=f(x,t),\qquad \nu=10^{-3},
\end{equation}
with $u(x,t)=e^{-t}q_\eps(x)$ and
\begin{equation}
  q_\eps(x)=\sin(\pi x)+0.20\sin\!\left(\frac{2\pi x}{\eps}\right).
\end{equation}
For $\eps=0.20,0.10,0.05$, the high-frequency component has modes $10,20,40$, respectively.  The values satisfy $2/\eps\in\mathbb N$, ensuring exact compatibility with the homogeneous boundary condition.  This harder approximation test uses the configuration in Table~\ref{tab:experiment_overview}.  Table~\ref{tab:epsdep1d} gives the twenty-seed errors.

\begin{table}[ht!]
\centering
\caption{One-dimensional $\eps$-dependent exact-solution benchmark.  Entries are twenty-seed final relative $L^2$ errors; the second line is the sample standard deviation.}
\label{tab:epsdep1d}
\tableformat
\begin{tabular}{@{}cccccc@{}}
\toprule
$\eps$ & TransNet & GTransNet & Gaussian RF & Uniform ELM & gain T/G\\
\midrule
0.20 & \meanstd{5.47\times10^{-6}}{3.52\times10^{-6}} & \meanstd{2.65\times10^{-6}}{9.00\times10^{-7}} & \meanstd{8.37\times10^{-5}}{7.71\times10^{-5}} & \meanstd{8.85\times10^{-5}}{6.36\times10^{-5}} & $2.1$\\
0.10 & \meanstd{1.43\times10^{-4}}{1.80\times10^{-4}} & \meanstd{1.59\times10^{-5}}{4.11\times10^{-6}} & \meanstd{2.97\times10^{-3}}{3.53\times10^{-3}} & \meanstd{2.75\times10^{-3}}{3.22\times10^{-3}} & $9.0$\\
0.05 & \meanstd{6.00\times10^{-3}}{5.44\times10^{-3}} & \meanstd{1.30\times10^{-4}}{3.26\times10^{-5}} & \meanstd{5.60\times10^{-2}}{2.67\times10^{-2}} & \meanstd{6.58\times10^{-2}}{3.20\times10^{-2}} & $46.2$\\
\bottomrule
\end{tabular}
\end{table}

At the reported hyperparameters and matched nominal output dimension, GTransNet gives the lowest mean error for all three values of $\eps$.  As the high mode increases from $10$ to $40$, the TransNet-to-GTransNet mean-error ratio increases from $2.1$ to $46.2$.  The corresponding median-error ratios are $1.6$, $6.6$, and $36.4$, supporting the same ordering without relying only on the means.

\subsubsection{Classical piecewise-linear finite-element baseline}

To place the fixed-feature errors beside a classical spatial discretization, Table~\ref{tab:p1_fem} reports a continuous piecewise-linear Galerkin method on a uniform mesh, with consistent mass, $16$-point elementwise Gauss integration, the same implicit midpoint step $\Delta t=0.0025$, and the same separate $4096$-point validation rule.  The number of interior finite-element unknowns equals the nominal evolved output dimension in the corresponding feature experiment.  The comparison therefore matches evolved dimension rather than computational cost: GTransNet additionally constructs $N_1$ first-layer features, whereas the one-dimensional finite-element matrices are sparse and structured.

\begin{table}[ht!]
\centering
\caption{Uniform continuous piecewise-linear finite-element baseline at the same nominal number of evolved unknowns.}
\label{tab:p1_fem}
\tableformat
\begin{tabular}{@{}cccc@{}}
\toprule
$\eps$ & interior DOFs & oscillatory coefficient & $\eps$-dependent target\\
\midrule
0.20 & $60/180$ & $6.70\times10^{-2}$ & $2.53\times10^{-4}$\\
0.10 & $60/180$ & $9.61\times10^{-2}$ & $1.64\times10^{-3}$\\
0.05 & $60/180$ & $6.56\times10^{-2}$ & $8.83\times10^{-3}$\\
\bottomrule
\end{tabular}
\end{table}

At these matched nominal dimensions, GTransNet has a lower error than the uniform P1 reference in all six cases.  TransNet also has a lower error in all six, with modest margins in the fixed-profile family.  This is a conventional low-order, uniform-mesh baseline; the feature and finite-element constructions have different computational costs.

\subsubsection{Feature-count sensitivity}

Table~\ref{tab:conv} varies $N$ on the $\eps=0.10$ oscillatory-coefficient benchmark.  This five-seed study illustrates dependence on feature count, while the twenty-seed comparison in Table~\ref{tab:main1d} provides the broader sample statistics.  All rows use the same $320$-point assembly rule, separate $4096$-point validation grid, and fixed $\gamma=10,20,8$ for TransNet, GTransNet, and Gaussian RF; only $N$ and the stated $N_1$ change.

\begin{table}[ht!]
\centering
\caption{Feature-count sensitivity at $\eps=0.10$.  Entries are five-seed final relative $L^2$ errors; the second line is the sample standard deviation.}
\label{tab:conv}
\tableformat
\begin{tabular}{@{}ccccc@{}}
\toprule
$N$ & $N_1$ & TransNet & GTransNet & Gaussian RF\\
\midrule
20 & 80  & \meanstd{1.16}{1.50\times10^{-1}} & \meanstd{1.15}{1.46\times10^{-1}} & \meanstd{1.32}{6.46\times10^{-2}}\\
30 & 120 & \meanstd{1.09}{4.54\times10^{-1}} & \meanstd{1.02\times10^{-1}}{5.90\times10^{-2}} & \meanstd{1.22}{3.01\times10^{-1}}\\
44 & 176 & \meanstd{3.51\times10^{-1}}{3.96\times10^{-1}} & \meanstd{2.21\times10^{-3}}{1.05\times10^{-3}} & \meanstd{1.01}{3.13\times10^{-1}}\\
60 & 240 & \meanstd{3.91\times10^{-2}}{2.33\times10^{-2}} & \meanstd{5.15\times10^{-4}}{3.48\times10^{-4}} & \meanstd{7.13\times10^{-1}}{4.13\times10^{-1}}\\
80 & 320 & \meanstd{5.78\times10^{-3}}{3.14\times10^{-3}} & \meanstd{1.66\times10^{-4}}{3.26\times10^{-5}} & \meanstd{3.03\times10^{-1}}{2.55\times10^{-1}}\\
\bottomrule
\end{tabular}
\end{table}

\subsection{Retained space and discretization diagnostics}

\subsubsection{Rank and conditioning}

Table~\ref{tab:diagnostics} records retained ranks, double-precision resolution of the full raw mass matrices, retained spectral ratios, and identity-mass defects for the oscillatory-coefficient benchmark.  For IEEE double precision, $\epsilon_{\rm mach}\approx2.22\times10^{-16}$; a run is unresolved when $\sigma_{\min}(\mathbf M_q)\le \epsilon_{\rm mach}\sigma_{\max}(\mathbf M_q)$.  The count in parentheses is the number of unresolved runs among the twenty samples; condition estimates are not assigned to unresolved raw matrices.

\begin{table}[ht!]
\centering
\caption{Rank and conditioning statistics for the oscillatory-coefficient benchmark.  ``Singular'' means numerically unresolved in double precision; $\kappa_{\rm ret}$ and $\varepsilon_M$ concern the retained quadrature-normalized space.}
\label{tab:diagnostics}
\tableformat
\begin{tabular}{@{}cclccc@{}}
\toprule
$\eps$ & Dictionary & $\bar r$ & raw $\mathbf M_q$ & med. $\kappa_{\rm ret}$ & max $\varepsilon_M$\\
\midrule
0.20 & TransNet    & 26.4 & singular (20) & $5.83\times10^{11}$ & $1.83\times10^{-10}$\\
     & GTransNet   & 48.5 & $\sim10^{15}$ (0) & $7.35\times10^{11}$ & $1.53\times10^{-10}$\\
     & Gaussian RF & 23.6 & singular (20) & $2.99\times10^{11}$ & $9.00\times10^{-11}$\\
     & Uniform ELM & 24.6 & singular (20) & $5.40\times10^{11}$ & $2.54\times10^{-10}$\\[2pt]
0.10 & TransNet    & 26.4 & singular (20) & $5.83\times10^{11}$ & $1.83\times10^{-10}$\\
     & GTransNet   & 48.5 & $\sim10^{15}$ (0) & $7.35\times10^{11}$ & $1.53\times10^{-10}$\\
     & Gaussian RF & 21.9 & singular (20) & $3.55\times10^{11}$ & $1.24\times10^{-10}$\\
     & Uniform ELM & 24.6 & singular (20) & $5.40\times10^{11}$ & $2.54\times10^{-10}$\\[2pt]
0.05 & TransNet    & 26.4 & singular (20) & $5.83\times10^{11}$ & $1.83\times10^{-10}$\\
     & GTransNet   & 51.6 & $\sim10^{14}$ (0) & $7.62\times10^{11}$ & $1.16\times10^{-10}$\\
     & Gaussian RF & 23.6 & singular (20) & $2.99\times10^{11}$ & $9.00\times10^{-11}$\\
     & Uniform ELM & 24.6 & singular (20) & $5.40\times10^{11}$ & $2.54\times10^{-10}$\\
\bottomrule
\end{tabular}
\end{table}

GTransNet retains a larger effective dimension after SVD rank selection than the one-layer baselines in this benchmark.  This is consistent with the interpretation that its variance-controlled second layer creates a richer retained span rather than changing the time integrator.  The raw-mass statuses in Table~\ref{tab:diagnostics} describe the untransformed dictionaries; time stepping uses the retained identity-mass system and never solves with these raw matrices.

\subsubsection{Truncation-threshold and quadrature sensitivity}

Table~\ref{tab:tau_diag} varies the truncation threshold on the $\eps=0.10$ oscillatory-coefficient benchmark using ten seeds and the $320$-point assembly rule.  Tighter thresholds retain useful directions and improve the error in this example, while increasing the retained rank, spectral sensitivity, and identity-mass defect.  The default $\tau=10^{-12}$ balances these effects; $\tau=10^{-8}$ discards useful directions and substantially increases the error.

\begin{table}[ht!]
\centering
\caption{Truncation-threshold sensitivity at $\eps=0.10$.  Error entries show the ten-seed mean with the sample standard deviation on the second line.}
\label{tab:tau_diag}
\tableformat
\begin{tabular}{@{}ccccc@{}}
\toprule
$\tau$ & TransNet error & GTransNet error & $\bar r_{\rm G}$ & max $\varepsilon_{M,\rm G}$\\
\midrule
$10^{-8}$  & \meanstd{8.91\times10^{-1}}{3.47\times10^{-1}} & \meanstd{2.17\times10^{-2}}{1.01\times10^{-2}} & 30.6 & $1.11\times10^{-12}$\\
$10^{-10}$ & \meanstd{1.98\times10^{-1}}{1.44\times10^{-1}} & \meanstd{3.01\times10^{-3}}{8.48\times10^{-4}} & 39.5 & $1.66\times10^{-11}$\\
$10^{-12}$ & \meanstd{4.53\times10^{-2}}{4.65\times10^{-2}} & \meanstd{4.47\times10^{-4}}{2.67\times10^{-4}} & 48.4 & $1.22\times10^{-10}$\\
$10^{-14}$ & \meanstd{2.42\times10^{-2}}{3.83\times10^{-2}} & \meanstd{1.09\times10^{-4}}{3.46\times10^{-5}} & 56.0 & $1.51\times10^{-9}$\\
\bottomrule
\end{tabular}
\end{table}

Quadrature sensitivity is assessed with $\tau=10^{-12}$ on the same separate validation grid.  The GTransNet mean error is $9.83\times10^{-4}$ at $Q_a=80$ and $4.47\times10^{-4}$ at $Q_a=120,160,240,$ and $320$ to the displayed precision; the TransNet mean remains $4.53\times10^{-2}$.  The plateau begins by $Q_a=120$, so the selected $Q_a=320$ lies within the resolved regime.

\subsubsection{Temporal convergence of the implicit midpoint scheme on a fixed retained space}

To isolate time discretization from feature approximation, we freeze the retained GTransNet space generated with seed 0, using $N=60$, $N_1=240$, $\gamma=20$, $\tau=10^{-12}$, and $320$ Gauss--Legendre assembly points.  In the autonomous test, the initial profile is $\sin(8\pi x)$, $\nu=10^{-2}$, and $T=0.25$; a symmetric eigendecomposition gives the exact solution of the fixed semidiscrete coefficient ODE.  In the nonautonomous test, the initial profile is $\sin(4\pi x)$ and
\[
  a(x,t)=10^{-2}\left[1+0.2\sin(2\pi t/0.1)\cos(2\pi x)\right],
  \qquad T=0.1.
\]
An adaptive eighth-order Dormand--Prince Runge--Kutta solution (DOP853) supplies the coefficient reference, with relative and absolute tolerances $10^{-12}$ and $10^{-14}$.  Tightening the tolerances to $10^{-13}$ and $10^{-15}$ and imposing maximum step $10^{-4}$ changes the final coefficient vector by $9.13\times10^{-14}$ relatively, more than six orders of magnitude below the smallest error in Table~\ref{tab:temporal_convergence}.  Because the time-dependent matrices generally do not commute, the second test is genuinely nonautonomous.

\begin{table}[ht!]
\centering
\caption{Coefficient-level temporal convergence on one fixed retained GTransNet space.  The Euclidean coefficient norm equals the assembly-quadrature $L^2$ norm in the normalized basis.}
\label{tab:temporal_convergence}
\tableformat
\begin{tabular}{@{}ccccc@{}}
\toprule
$\Delta t$ & autonomous error & EOC & nonautonomous error & EOC\\
\midrule
$10^{-2}$       & $5.25\times10^{-4}$ & --    & $1.23\times10^{-5}$ & --\\
$5\times10^{-3}$   & $1.31\times10^{-4}$ & 2.000 & $3.04\times10^{-6}$ & 2.012\\
$2.5\times10^{-3}$ & $3.28\times10^{-5}$ & 2.000 & $7.59\times10^{-7}$ & 2.003\\
$1.25\times10^{-3}$& $8.20\times10^{-6}$ & 2.000 & $1.90\times10^{-7}$ & 2.001\\
\bottomrule
\end{tabular}
\end{table}

The observed EOCs are approximately two and agree with the second-order accuracy of the implicit midpoint update on the fixed retained space.

\subsection{Causal coefficient evolution versus global space--time regression}

We next isolate the role of causal coefficient evolution by comparing it with a global regression that treats time as another input coordinate.  For the heat equation
\begin{equation}
  u_t=10^{-2}u_{xx},\qquad u(0,t)=u(1,t)=0,
  \qquad u(x,0)=\sin(k\pi x),
\end{equation}
consider the global ansatz
\begin{equation}
  u_{\rm st}(x,t)=u_0(x)+tD(x)\vpsi(x,t/T)^\top\valpha.
\end{equation}
It satisfies the initial and boundary conditions exactly but does not enforce causal time marching or parabolic energy decay.  The residual is sampled without additional row weights at $68$ uniformly spaced interior points and $69$ positive time levels, giving $\mathbf A_{\rm st}\in\mathbb R^{4692\times60}$.  The regression minimizes $\|\mathbf A_{\rm st}\valpha-\vb_{\rm st}\|_2^2+\rho\|\valpha\|_2^2$, where $\rho=10^{-10}\operatorname{tr}(\mathbf A_{\rm st}^\top\mathbf A_{\rm st})/N$.  We solve the equivalent column-scaled augmented least-squares problem by an SVD-based solver; the largest observed condition estimate of the augmented scaled matrix is below $500$.

Both formulations use $N=60$, $N_1=240$, $\gamma=8$, seeds $0,1,2$, $T=0.25$, and the same separate $1000$-point final-time validation rule.  The output-evolution calculation uses $320$-point assembly and $\Delta t=0.0025$.  Because the methods employ different trial spaces and residual weights, Table~\ref{tab:global_diagnostic} is a diagnostic comparison of these specified formulations rather than a cost-matched ranking.

\begin{table}[ht!]
\centering
\caption{Global space--time regression and causal output evolution for dissipative heat modes.  Entries are three-seed final relative errors; the second line is the sample standard deviation.}
\label{tab:global_diagnostic}
\tableformat
\begin{tabular}{@{}ccc@{}}
\toprule
$k$ & global space--time regression & causal output evolution\\
\midrule
8  & \meanstd{2.34\times10^{-1}}{2.95\times10^{-1}} & \meanstd{5.14\times10^{-5}}{1.61\times10^{-5}}\\
12 & \meanstd{3.30\times10^{1}}{2.88\times10^{-1}} & \meanstd{7.40\times10^{-4}}{1.69\times10^{-4}}\\
16 & \meanstd{5.54\times10^{2}}{2.57} & \meanstd{9.60\times10^{-3}}{3.13\times10^{-3}}\\
\bottomrule
\end{tabular}
\end{table}

Under these settings, the global ansatz retains large final errors even with the well-conditioned least-squares solve, whereas causal output evolution remains accurate at the same nominal feature dimensions.

\subsection{Nonautonomous and multidimensional benchmarks}

\subsubsection{Two-dimensional fixed and time-dependent diffusion operators}

We first consider the equation
\begin{equation}
  u_t-\nabla\cdot(a_\eps(x,y)\nabla u)=f(x,y,t),\qquad (x,y)\in(0,1)^2,
\end{equation}
where
\begin{equation}
  a_\eps(x,y)=1+\frac14\cos\left(\frac{2\pi x}{\eps}\right)
  \cos\left(\frac{2\pi y}{\eps}\right).
  \label{eq:2d_fixed_coeff}
\end{equation}
The exact solution is $u(x,y,t)=e^{-t}\sin(6\pi x)\sin(6\pi y)$.  The companion nonautonomous problem replaces the coefficient by
\begin{equation}
  a_\eps(x,y,t)=1+\frac14\left(1+\frac12\sin\left(\frac{2\pi t}{T}\right)\right)
  \cos\left(\frac{2\pi x}{\eps}\right)
  \cos\left(\frac{2\pi y}{\eps}\right).
  \label{eq:2d_time_coeff}
\end{equation}
The exact solution is unchanged.  The time-dependent stiffness matrix is reassembled at every implicit-midpoint step.  Both problems use the common configuration in Table~\ref{tab:experiment_overview}.

\begin{table}[ht!]
\centering
\caption{Two-dimensional fixed and time-dependent diffusion operators at $T=0.06$.  Entries are ten-seed final relative $L^2$ errors; the second line is the sample standard deviation.}
\label{tab:2d_operator_comparison}
\tableformat
\begin{tabular}{@{}cccccc@{}}
\toprule
$\eps$ & operator & TransNet & GTransNet & Gaussian RF & gain T/G\\
\midrule
0.20 & fixed & \meanstd{7.96\times10^{-2}}{2.61\times10^{-2}} & \meanstd{2.05\times10^{-2}}{3.36\times10^{-3}} & \meanstd{2.54\times10^{-1}}{3.29\times10^{-2}} & $3.9$\\
     & time dep. & \meanstd{7.94\times10^{-2}}{2.60\times10^{-2}} & \meanstd{2.05\times10^{-2}}{3.36\times10^{-3}} & \meanstd{2.54\times10^{-1}}{3.28\times10^{-2}} & $3.9$\\
0.10 & fixed & \meanstd{7.90\times10^{-2}}{2.59\times10^{-2}} & \meanstd{2.04\times10^{-2}}{3.35\times10^{-3}} & \meanstd{2.53\times10^{-1}}{3.27\times10^{-2}} & $3.9$\\
     & time dep. & \meanstd{7.90\times10^{-2}}{2.59\times10^{-2}} & \meanstd{2.04\times10^{-2}}{3.35\times10^{-3}} & \meanstd{2.53\times10^{-1}}{3.27\times10^{-2}} & $3.9$\\
\bottomrule
\end{tabular}
\end{table}

The sinusoidal modulation has zero temporal mean over $[0,T]$, and both benchmarks use the same exact spatial profile.  Thus the comparison isolates midpoint reassembly of $\widetilde{\mathbf K}_q(t)$ rather than introducing a new representation challenge.  The relative change in mean error is at most $0.31\%$, and the error ordering at matched nominal $N$ is unchanged.

\subsubsection{A two-dimensional nonseparable space--time benchmark}

To move beyond solutions of the separable form $u(\bx,t)=c(t)q(\bx)$, we prescribe
\begin{equation}
  u(x,y,t)=\sum_{m=1}^3 A_m e^{-\lambda_m t}
  \sin(k_{x,m}\pi x)\sin(k_{y,m}\pi y),
\end{equation}
with $(k_{x,m},k_{y,m},A_m,\lambda_m)=(2,3,1,1),(6,5,0.45,4),(9,7,0.25,7)$.  The diffusion coefficient is \eqref{eq:2d_time_coeff}.  We use the configuration in Table~\ref{tab:experiment_overview}.  Table~\ref{tab:nonsep2d} reports the final errors, and Figure~\ref{fig:nonsep2d_fields} shows a representative final field and error map.

\begin{table}[ht!]
\centering
\caption{Two-dimensional nonseparable benchmark.  Entries are ten-seed final relative $L^2$ errors; the second line is the sample standard deviation.}
\label{tab:nonsep2d}
\tableformat
\begin{tabular}{@{}ccccc@{}}
\toprule
$\eps$ & TransNet & GTransNet & Gaussian RF & gain T/G\\
\midrule
0.20 & \meanstd{4.88\times10^{-2}}{7.04\times10^{-3}} & \meanstd{1.35\times10^{-2}}{2.76\times10^{-3}} & \meanstd{6.97\times10^{-2}}{1.47\times10^{-2}} & $3.6$\\
0.10 & \meanstd{4.81\times10^{-2}}{6.92\times10^{-3}} & \meanstd{1.33\times10^{-2}}{2.65\times10^{-3}} & \meanstd{6.90\times10^{-2}}{1.45\times10^{-2}} & $3.6$\\
\bottomrule
\end{tabular}
\end{table}

\begin{figure}[H]
\centering
\includegraphics[width=\textwidth]{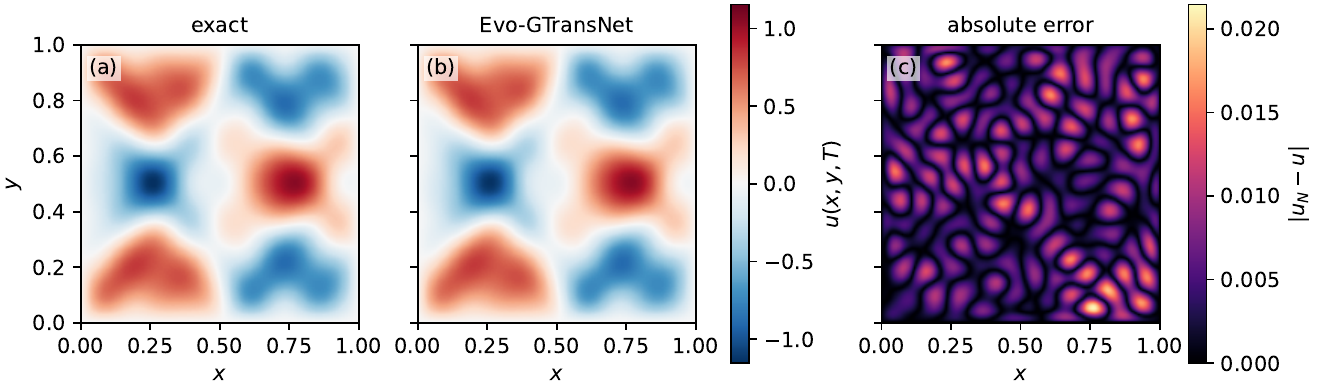}
\caption{Representative final-time field for the two-dimensional nonseparable benchmark at $\eps=0.10$ using the seed-zero GTransNet dictionary.  The exact and numerical panels use a common color scale; the right panel shows the pointwise absolute error.}
\label{fig:nonsep2d_fields}
\end{figure}

Unlike the fixed-profile test, the modal weights now decay at different rates, so the output coefficients follow a multidimensional trajectory rather than a scalar rescaling of one spatial profile.

\subsubsection{A three-dimensional nonseparable parabolic box benchmark}

On $\Omega=(0,1)^3$, let us set
\begin{equation}
  u(x,y,z,t)=\sum_{m=1}^3 A_m e^{-\lambda_m t}
  \sin(k_{x,m}\pi x)\sin(k_{y,m}\pi y)\sin(k_{z,m}\pi z),
\end{equation}
with mode parameters
\[
\begin{array}{c|ccccc}
m & k_{x,m} & k_{y,m} & k_{z,m} & A_m & \lambda_m\\
\hline
1 & 1 & 2 & 2 & 1 & 1\\
2 & 3 & 4 & 2 & 0.45 & 3\\
3 & 5 & 3 & 4 & 0.25 & 5
\end{array}
\]
and
\begin{equation}
  a_\eps(x,y,z,t)=1+0.16\left(1+\frac12\sin\left(\frac{2\pi t}{T}\right)\right)
  \prod_{\xi\in\{x,y,z\}}\cos\left(\frac{2\pi \xi}{\eps}\right).
\end{equation}
We use the configuration in Table~\ref{tab:experiment_overview}; Table~\ref{tab:nonsep3d} reports the results.

\begin{table}[ht!]
\centering
\caption{Three-dimensional nonseparable parabolic box benchmark.  Entries are five-seed final relative $L^2$ errors; the second line is the sample standard deviation.}
\label{tab:nonsep3d}
\tableformat
\begin{tabular}{@{}ccccc@{}}
\toprule
$\eps$ & TransNet & GTransNet & Gaussian RF & gain T/G\\
\midrule
0.25  & \meanstd{5.02\times10^{-2}}{7.10\times10^{-3}} & \meanstd{2.74\times10^{-2}}{2.85\times10^{-3}} & \meanstd{7.64\times10^{-2}}{4.98\times10^{-3}} & $1.8$\\
0.125 & \meanstd{5.02\times10^{-2}}{7.10\times10^{-3}} & \meanstd{2.74\times10^{-2}}{2.85\times10^{-3}} & \meanstd{7.64\times10^{-2}}{4.98\times10^{-3}} & $1.8$\\
\bottomrule
\end{tabular}
\end{table}

At the stated hyperparameters and matched nominal $N=320$, GTransNet reduces the mean final error by factors of approximately $1.8$ and $2.8$ relative to TransNet and Gaussian RF, respectively.  The nearly identical rows indicate that spatial approximation of the evolving modal content dominates the final error over the tested parameter range.

\subsection{Nonlinear extension: A two-dimensional Allen--Cahn equation}

The preceding analysis concerns linear parabolic equations.  To test the retained spatial discretization in a nonlinear setting, we combine it with a standard stabilized dissipative step for the unforced Allen--Cahn equation \cite{allen1979ac} defined as
\begin{equation}
  u_t=\varepsilon_{\rm ac}^2\Delta u+u-u^3,\qquad
  u|_{\partial(0,1)^2}=0,
  \label{eq:ac_test}
\end{equation}
where $\varepsilon_{\rm ac}=0.025$ and
\begin{align}
  u_0(x,y)=&
  0.42\sin(6\pi x)\sin(6\pi y)
  +0.24\sin(4\pi x)\sin(7\pi y) \\
  &+0.16\sin(9\pi x)\sin(3\pi y).
\end{align}
The stabilized semi-implicit discretization \cite{shen2010allen} is
\begin{equation}
  \frac{u_N^{n+1}-u_N^n}{\Delta t}
  =\varepsilon_{\rm ac}^2\Delta u_N^{n+1}
  +u_N^n-(u_N^n)^3-S(u_N^{n+1}-u_N^n),
  \label{eq:ac_scheme}
\end{equation}
where $S\ge0$.  In quadrature-mass-orthonormal coordinates,
\begin{equation}
  \left((1+S\Delta t)\mathbf I_r
  +\Delta t\,\varepsilon_{\rm ac}^2\widetilde{\mathbf K}_q\right)\valpha^{n+1}
  =(1+S\Delta t)\valpha^n
  +\Delta t\,\boldsymbol g(\valpha^n),
\end{equation}
with
\begin{equation}
  \bigl(\boldsymbol g(\valpha^n)\bigr)_i
  =\sum_{q=1}^Qw_q
   \left[u_{N,r}^n(\bx_q)-\bigl(u_{N,r}^n(\bx_q)\bigr)^3\right]
   (\widetilde{\bm\Phi})_{qi}.
  \label{eq:ac_g_quad}
\end{equation}
Here $\widetilde{\mathbf K}_q$ denotes the retained quadrature-assembled stiffness matrix corresponding to $a\equiv1$. 

The feature solution uses $S=2$ and the configuration in Table~\ref{tab:experiment_overview}.  Its assembly-quadrature energy is
\begin{equation}
  E_q(u)=\sum_qw_q\left[\frac{\varepsilon_{\rm ac}^2}{2}|\nabla u(\bx_q)|^2
  +\frac14(u(\bx_q)^2-1)^2\right].
\end{equation}
The reference solution is computed on a $160\times160$ interior finite-difference grid using the same stabilized step, with each linear solve diagonalized by the discrete sine transform.  A grid- and time-step-refinement study gives final relative differences $1.41\times10^{-3}$ between the $160^2$ and $320^2$ grids at fixed step, $7.54\times10^{-4}$ when the step is halved on the $320^2$ grid, and $2.05\times10^{-3}$ between the reference and the fully refined $320^2$, $\Delta t=0.001$ solution.  These differences lie below the feature errors.  Final reference errors are evaluated on the separate $60\times60$ validation rule after bilinear interpolation of the zero-boundary finite-difference field.

\begin{table}[ht!]
\centering
\caption{Two-dimensional Allen--Cahn test.  Error entries show five-seed means with sample standard deviations on the second line; rank and energy drop are seed means.  Negative $\max_n\Delta E_q^n$ indicates monotone observed decay.}
\label{tab:allen_cahn}
\tableformat
\begin{tabular}{@{}cccccc@{}}
\toprule
Dictionary & final error & projection error & $\bar r$ & energy drop & $\max_n\Delta E_q^n$\\
\midrule
TransNet & \meanstd{1.27\times10^{-1}}{3.57\times10^{-2}} & \meanstd{1.08\times10^{-1}}{4.39\times10^{-2}} & 99.8 & 4.34\% & $-1.92\times10^{-5}$\\
GTransNet & \meanstd{5.61\times10^{-2}}{3.51\times10^{-3}} & \meanstd{2.82\times10^{-2}}{4.60\times10^{-3}} & 150 & 4.53\% & $-2.00\times10^{-5}$\\
Gaussian RF & \meanstd{2.85\times10^{-1}}{5.47\times10^{-2}} & \meanstd{2.77\times10^{-1}}{5.29\times10^{-2}} & 98.8 & 3.91\% & $-1.65\times10^{-5}$\\
\bottomrule
\end{tabular}
\end{table}

\begin{figure}[H]
\centering
\includegraphics[width=0.62\textwidth]{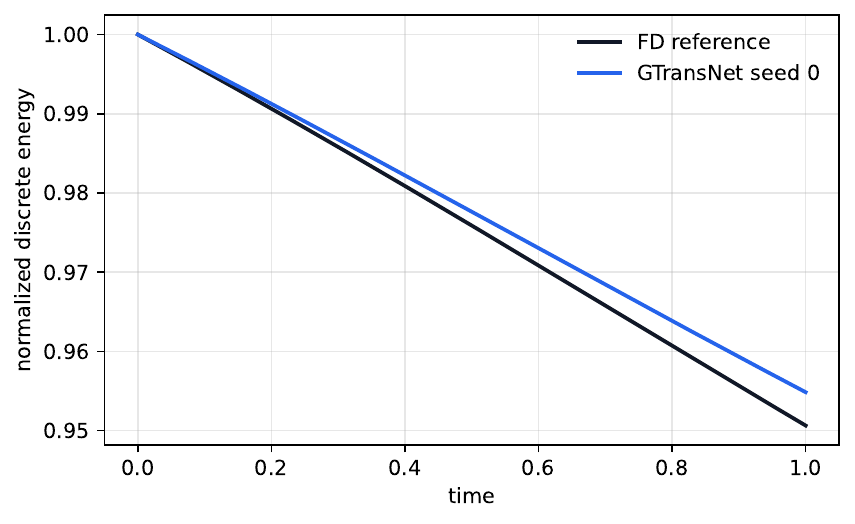}
\caption{Normalized discrete energy histories for the finite-difference reference and one representative GTransNet run.  Both decay monotonically under the stabilized semi-implicit step.}
\label{fig:allen_cahn}
\end{figure}

The retained-space discretization is compatible with the stabilized nonlinear step: all tested dictionaries show monotone observed energy decay.  At the stated hyperparameters and matched nominal $N=150$, GTransNet has the smallest mean initial-projection and final-reference errors.

\section{Conclusions}\label{sec:conclusion}

In this paper, we have developed Evo-GTransNet for parabolic equations. It is a Galerkin method of lines that evolves the coefficients in a fixed spatial trial space supplied by a prescribed GTransNet dictionary. A quadrature-weighted truncated SVD selects the numerically resolved subspace before time integration. The selected basis is then rescaled so that its mass matrix is the identity with respect to the assembly quadrature. The two operations play different roles: truncation determines the approximation space, whereas rescaling alone leaves the retained discrete function unchanged in exact arithmetic.

For symmetric linear parabolic problems, the semidiscrete formulation satisfies the Galerkin energy law, and the implicit midpoint discretization in time is contractive in the assembly-quadrature norm. The conditional fully discrete estimate relates the error to the retained-space approximation error and to the defects arising from spatial assembly, initialization, and temporal consistency. The numerical experiments recover second-order temporal convergence for both autonomous and nonautonomous problems and quantify the behavior of the retained space when the raw feature mass matrices are severely ill-conditioned. For the high-frequency and multiscale benchmarks considered here, GTransNet gives the smallest mean errors among the tested dictionaries at the stated hyperparameter settings and matched nominal output dimension. The multidimensional tests and the observed energy decay in the stabilized Allen--Cahn experiment extend the numerical study beyond one-dimensional linear problems.

Once a dictionary has been sampled, the time integration evolves only the retained output coefficients and cannot recover spatial directions absent from the trial space. The numerical results therefore reflect the combined effects of the trial-space approximation, the retained rank, the sampling variability, and the feature construction. Since GTransNet also introduces the auxiliary width $N_1$, the comparisons at matched nominal output dimension evaluate spatial approximation rather than overall computational efficiency.
The present study is limited to fixed global dictionaries and moderate-size parabolic benchmarks; the multidimensional examples do not address large-scale computational complexity. Future work includes probabilistic approximation estimates, adaptive or localized feature spaces, and scalable rank-revealing algorithms. Extending the analysis to complex geometries, heterogeneous media, and broader classes of nonlinear or nonsymmetric evolution equations is another natural direction.

\appendix

\section{Feature-derivative formulas}\label{app:feature_derivatives}

This appendix records the feature-derivative formulas used in the numerical implementation.
Let $\sigma(z)=\tanh(z)$. For any scalar preactivation $z=z(\bx)$ and $\psi=\sigma(z)$, the chain rule gives
\begin{align}
  \partial_{x_\ell}\psi
  &=\sigma'(z)\,\partial_{x_\ell}z,\\
  \partial_{x_\ell x_m}^2\psi
  &=\sigma''(z)(\partial_{x_\ell}z)(\partial_{x_m}z)
    +\sigma'(z)\,\partial_{x_\ell x_m}^2z,
  \label{eq:tanh_chain_compact}
\end{align}
where
 $ \sigma'(z)=1-\sigma(z)^2$ and
  $\sigma''(z)=-2\sigma(z)\bigl(1-\sigma(z)^2\bigr)$.
For a one-layer feature, $z_j=\gamma_j(\ba_j^\top(\bx-\bx_c)+R\xi_j)$ is affine, and therefore $\partial_{x_\ell x_m}^2z_j=0$. For a deeper GTransNet layer,
\begin{equation}
  \psi_i^{[k]}=\sigma(z_i^{[k]}),
  \qquad
  z_i^{[k]}=\sum_j(\mathbf W^{[k]})_{ij}\psi_j^{[k-1]},
  \qquad k\ge2,
\end{equation}
the first and second derivatives are propagated recursively by linearity and \eqref{eq:tanh_chain_compact}.

If the hard-boundary basis function is $\varphi=D\psi$, then
\begin{align}
  \partial_{x_\ell}\varphi
  &=(\partial_{x_\ell}D)\psi+D\partial_{x_\ell}\psi,\\
  \partial_{x_\ell x_m}^2\varphi
  &=(\partial_{x_\ell x_m}^2D)\psi
    +(\partial_{x_\ell}D)(\partial_{x_m}\psi)
    +(\partial_{x_m}D)(\partial_{x_\ell}\psi)
    +D\partial_{x_\ell x_m}^2\psi.
\end{align}
The Galerkin assembly uses only the first spatial derivatives. The second derivatives are used for the global space--time strong-residual diagnostic and for independent derivative verification.

\section{Numerical assembly and reproducibility details}\label{app:assembly}
\subsection{Shape-parameter candidate sets}\label{app:gamma_candidates}

For all four dictionaries, $\gamma$ is the common multiplier in the argument of the first-layer $\tanh$ features; for GTransNet, it is the scalar in $\mathbf\Gamma=\gamma\mathbf I_{N_1}$. Thus, the numbers below are candidate values of the shape parameter $\gamma$, not feature counts, layer widths, quadrature sizes, or random seeds. We write $\mathcal G_m^b$ for the fixed candidate set associated with dictionary $m\in\{\mathrm T,\mathrm G,\mathrm{RF},\mathrm{ELM}\}$ and benchmark family $b\in\{\mathrm{osc},\mathrm{tar}\}$. The subscripts denote TransNet, GTransNet, Gaussian RF, and Uniform ELM, while the superscripts denote the oscillatory-coefficient and $\eps$-dependent-target benchmarks, respectively. The candidate sets are as follows. For the oscillatory-coefficient benchmark,
\begin{align*}
\mathcal G_{\rm T}^{\rm osc}&=\{4,6,8,10,12,14,16,18\},\\
\mathcal G_{\rm G}^{\rm osc}&=\{6,8,10,12,14,16,18,20,22,24,26,28,32\},\\
\mathcal G_{\rm RF}^{\rm osc}&=\{2,4,6,8,10,12,14\},\\
\mathcal G_{\rm ELM}^{\rm osc}&=\{2,4,6,8,10,12,14,16,18,20,24,28\}.
\end{align*}
For the $\eps$-dependent exact-solution benchmark,
\begin{align*}
\mathcal G_{\rm T}^{\rm tar}&=\{10,14,18,22,26,30,34,38,42,48\},\\
\mathcal G_{\rm G}^{\rm tar}&=\{12,18,24,30,36,42,48,54,60\},\\
\mathcal G_{\rm RF}^{\rm tar}&=\{8,12,16,20,24,30,36,42,48\},\\
\mathcal G_{\rm ELM}^{\rm tar}&=\{8,12,16,20,24,30,36,42,48,54\}.
\end{align*}

Let $\mathcal S_p=\{100,101,102,103,104\}$ be the pilot-seed set, and let $E_{m,\eps}^{(s)}(\gamma)$ denote the final relative $L^2$ error in \eqref{eq:validation_error}, evaluated on the separate validation grid for dictionary $m$, seed $s$, and parameter $\eps$. For each of the two one-dimensional benchmark families, each dictionary, and each reported value of $\eps$, we compute
\[
  \overline E_{m,\eps}(\gamma)
  :=\frac{1}{|\mathcal S_p|}\sum_{s\in\mathcal S_p}
       E_{m,\eps}^{(s)}(\gamma),
  \qquad
  \gamma_{m,\eps}^{*}
  :=\min\!\left(
      \operatorname*{arg\,min}_{\gamma\in\mathcal G_m^b}
      \overline E_{m,\eps}(\gamma)
    \right),
\]
where $b=\mathrm{osc}$ or $\mathrm{tar}$ according to the benchmark family. The outer minimum applies the stated tie rule by choosing the smaller $\gamma$ when the pilot means agree. The selected $\gamma_{m,\eps}^{*}$ is then fixed for the disjoint evaluation seeds $0$--$19$; no evaluation seed is used for parameter selection. Table~\ref{tab:gamma_selected} reports the selected values together with the prescribed values used in the remaining experiments.

\begin{table}[ht!]
\centering
\caption{Shape parameters used in the reported experiments. Each entry is the value of $\gamma$ for the corresponding dictionary. The one-dimensional benchmark values are selected on pilot seeds $100$--$104$; diagnostic and higher-dimensional values are prescribed and shared across the reported seeds and values of $\eps$, where applicable. A dash denotes an unused dictionary.}
\label{tab:gamma_selected}
\tableformat
\begin{tabular}{@{}lccccc@{}}
\toprule
Benchmark & $\eps$ & $\gamma_{\rm T}$ & $\gamma_{\rm G}$ & $\gamma_{\rm RF}$ & $\gamma_{\rm ELM}$\\
\midrule
1D global space--time diagnostic & -- & -- & 8 & -- & --\\
1D oscillatory coefficient & 0.20 & 10 & 20 & 10 & 20\\
                            & 0.10 & 10 & 20 & 8  & 20\\
                            & 0.05 & 10 & 22 & 10 & 20\\
1D $\eps$-dependent target & 0.20 & 14 & 24 & 12 & 12\\
                            & 0.10 & 18 & 30 & 16 & 24\\
                            & 0.05 & 30 & 48 & 36 & 42\\
Implicit-midpoint convergence & -- & -- & 20 & -- & --\\
\midrule
2D fixed/time-dependent operator & -- & 4 & 6 & 4 & --\\
2D nonseparable evolution        & -- & 5 & 8 & 5 & --\\
3D nonseparable evolution        & -- & 3.5 & 3.5 & 2.5 & --\\
2D Allen--Cahn                   & -- & 4 & 6 & 4 & --\\
\bottomrule
\end{tabular}
\end{table}

\subsection{Eigenvalue threshold}\label{app:threshold}

The retained rank is determined by the criterion
$\lambda_i>\tau\lambda_{\max}$ for the eigenvalues of
$\mathbf M_q=\bm\Phi^\top\mathbf W_q\bm\Phi$. Equivalently, if
$\mathbf B=\mathbf W_q^{1/2}\bm\Phi$ has singular values $\sigma_i$, then the retained directions satisfy
$\sigma_i>\sqrt{\tau}\,\sigma_1$. The computations use this SVD criterion directly. Unless otherwise stated, the reported experiments use $\tau=10^{-12}$; Table~\ref{tab:tau_diag} reports the sensitivity to this choice. We define the identity-mass defect by
\begin{equation}
  \varepsilon_M
  :=\norm{\widetilde{\bm\Phi}^\top\mathbf W_q
           \widetilde{\bm\Phi}-\mathbf I_r}_2.
\end{equation}
In exact arithmetic, $\varepsilon_M=0$ on the retained SVD subspace. In floating-point arithmetic, the SVD construction is more reliable than explicitly forming eigenvectors of a nearly singular mass matrix. All reported identity-mass defects are below $1.6\times10^{-9}$, including the sensitivity runs with $\tau=10^{-14}$. Thus, the retained bases have identity mass to the reported accuracy.

\subsection{Cross-grid quadrature assessment}\label{app:cross_grid}

The identity-mass relation is exact on the assembly rule used to construct $\mathbf T_r$; by itself, it does not measure quadrature accuracy away from that rule. For the cross-grid assessment, each assembly-selected transform is fixed, and the retained basis and its derivatives are reevaluated on a finer, separate tensor-product rule. We report
\begin{equation}
\varepsilon_M^{\rm cg}
:=\norm{\widetilde{\bm\Phi}_a^\top\mathbf W_a
\widetilde{\bm\Phi}_a-\mathbf I_r}_2,
\qquad
\varepsilon_K^{\rm cg}
:=\frac{\norm{\widetilde{\mathbf K}_a-\widetilde{\mathbf K}_q}_2}{\norm{\widetilde{\mathbf K}_q}_2},
\end{equation}
where the subscript $a$ denotes the separate assessment rule, and $\widetilde{\mathbf K}_a$ is obtained by assembling the stiffness form on that rule using the same fixed retained basis. Across the tested seeds, dictionaries, and coefficient configurations in the two- and three-dimensional tests, the largest observed values are $7.09\times10^{-3}$ for $\varepsilon_M^{\rm cg}$ and $7.01\times10^{-3}$ for $\varepsilon_K^{\rm cg}$.

For each two-dimensional linear benchmark, the configuration with the largest stiffness discrepancy is recomputed using a refined assembly rule. The relative changes in the final error are $4.636\times10^{-5}$ for the fixed-operator test, $4.580\times10^{-5}$ for the time-dependent-operator test, and $1.409\times10^{-5}$ for the nonseparable test. For the seed-zero three-dimensional Gaussian-RF sentinel, changing the assembly rule from $20^3$ to $24^3$ changes the final error by at most $7.913\times10^{-6}$ over the two reported values of $\eps$. These changes are much smaller than the corresponding final errors, indicating that the reported comparisons are not sensitive to the tested quadrature refinements. A uniform quadrature-convergence analysis is not included here.

\subsection{Computational scaling}\label{app:scaling}

Let $Q$ be the number of quadrature points, $d$ the spatial dimension, $N$ the final output dimension before SVD rank selection, $N_1$ the first hidden-layer width of GTransNet, and $r$ the retained rank after truncation. Direct evaluation of a two-layer GTransNet dictionary costs approximately $O(QN_1+QN_1N)$. For a one-layer random-feature dictionary, the $O(QN_1N)$ second-layer cost is absent. After the weighted feature and derivative-value matrices are formed, a thin SVD of $\mathbf W_q^{1/2}\bm\Phi$ costs $O(QN^2)$ when $Q\ge N$.

After rank selection, the evolved system has dimension $r$, and the quadrature-mass orthonormalization gives identity-mass coordinates. Once
$\widetilde{\mathbf G}_\ell\in\mathbb R^{Q\times r}$ is available, a scalar diffusion operator can be assembled as
\begin{equation}
  \widetilde{\mathbf K}_q(t)
  =\sum_{\ell=1}^d
   \widetilde{\mathbf G}_\ell^\top\mathbf W_q
   \mathbf D_{a,q}(t)\widetilde{\mathbf G}_\ell,
\end{equation}
at a cost of $O(dQr^2)$ for dense assembly. For a time-independent operator, the matrix
$\mathbf I_r+(\Delta t/2)\widetilde{\mathbf K}_q$ is factorized once at a cost of $O(r^3)$, and each later time step costs $O(r^2)$ after the right-hand side is assembled. For a time-dependent operator, the stiffness matrix and dense factorization are recomputed at each step, with cost $O(dQr^2+r^3)$ per step. The Allen--Cahn experiment also requires an explicit nonlinear projection. Its main cost is the evaluation at the quadrature points and projection onto the retained basis, approximately $O(Qr)$. These estimates are for dense algebra at the moderate retained ranks used in this paper. Sparse and iterative implementations are not considered here.

\section*{Declarations}

\subsection*{Funding}
L.~J. was supported in part by the U.S. Department of Energy under grant No.~DE-SC0022254.  J.~Z. was supported in part by the Beijing Natural Science Foundation (No.~JR25003).

\subsection*{Competing interests}
The authors declare that they have no known competing financial interests or personal relationships that could have appeared to influence the work reported in this paper.

\subsection*{Data and code availability}
No external data were used.  The code used to generate the reported numerical experiments is provided in the accompanying Supplementary Material; it specifies the parameter settings, random seeds, and quadrature rules used in this study.

\subsection*{Use of generative AI and AI-assisted technologies}
During the preparation of this work, the authors used OpenAI's ChatGPT and Codex to assist with language editing, consistency checks, and manuscript presentation.  The authors reviewed and verified the manuscript in full and take full responsibility for its content.

\bibliographystyle{plainnat}
\bibliography{references}

\end{document}